%% file: ChangePointRegression.tex
\documentclass[twoside,11pt]{article}

\usepackage{jmlr2e}
\usepackage{amsmath,bbm}
\usepackage{hyperref}
\usepackage{algorithm,amssymb}
\usepackage{algpseudocode}
\usepackage[dvipsnames]{xcolor}
\usepackage{tikz}
\usepackage{multirow}

\usetikzlibrary{arrows.meta}
\usetikzlibrary{decorations.pathreplacing}

\algnewcommand\algorithmicinput{\textbf{INPUT:}}
\algnewcommand\INPUT{\item[\algorithmicinput]}
\algnewcommand\algorithmicoutput{\textbf{OUTPUT:}}
\algnewcommand\OUTPUT{\item[\algorithmicoutput]}

\usepackage{cleveref}

\usepackage{cleveref}

\allowdisplaybreaks

\DeclareMathOperator*{\argmin}{arg\,min}

\newtheorem{model}{Model}

\newcommand{\eps}{\varepsilon}

\newcommand{\s}{{\mathfrak{s} }}
\newcommand{\N}{{\mathfrak{N} }}
\newcommand{\op}{{\text{op} }}
\newcommand{\kse}{{\kappa_{\max}^{s,e}}}

 \usepackage{lastpage}
\jmlrheading{22}{2021}{1-\pageref{LastPage}}{6/19; Revised
7/21}{10/21}{19-531} {Daren Wang, Zifeng Zhao, Kevin Z.  Lin   and Rebecca Willett}

\ShortHeadings{Regression Change Point Detection in High-dimension}{Wang, Zhao, Lin, Willett}
\firstpageno{1}

\begin{document}

\title{Statistically and Computationally Efficient Change Point Localization in Regression Settings}

\author{\name Daren Wang \email dwang24@nd.edu \\
       \addr Department of ACMS\\
       University of Notre Dame
\\Indiana, USA
       \AND
       \name Zifeng Zhao \email zifeng.zhao@nd.edu \\
       \addr Mendoza College of Business\\
        University of Notre Dame
\\Indiana, USA     
       \AND
	   \name Kevin Z. Lin \email kevinl1@wharton.upenn.edu \\	
	   \addr Department of Statistics and Data Science  \\
	   University of Pennsylvania\\
	   Pennsylvania, USA
	   \AND Rebecca Willett \email willett@uchicago.edu \\
	   \addr Department of Statistics  \\
	   University of Chicago\\
	     Illinois, USA}
\editor{Zaid Harchaoui} 
\maketitle

\begin{abstract}
Detecting when the underlying distribution changes for the observed time series is a fundamental problem arising in a broad spectrum of applications. In this paper, we study multiple change-point localization in the high-dimensional regression setting, which is particularly challenging as no direct observations of the parameter of interest is available. Specifically, we assume we observe $\{ x_t, y_t\}_{t=1}^n$ where $ \{ x_t\}_{t=1}^n $ are $p$-dimensional covariates, $\{y_t\}_{t=1}^n$ are the univariate responses satisfying $\mathbb{E}(y_t) = x_t^\top \beta_t^*  \text{ for } 1\le t \le n $ and $\{\beta_t^*\}_{t=1}^n $ are the  unobserved regression coefficients that change over time in a piecewise constant manner. We propose a novel projection-based algorithm, Variance Projected Wild Binary Segmentation~(VPWBS), which transforms the original (difficult) problem of change-point detection in $p$-dimensional regression to a simpler problem of change-point detection in mean of a one-dimensional time series. VPWBS is shown to achieve sharp localization rate $O_p(1/n)$ up to a log factor, a significant improvement from the best rate $O_p(1/\sqrt{n})$ known in the existing literature for multiple change-point localization in high-dimensional regression. Extensive numerical experiments are conducted to demonstrate the robust and favorable performance of VPWBS over two state-of-the-art algorithms, especially when the size of change in the regression coefficients $\{\beta_t^*\}_{t=1}^n $ is small.
\end{abstract}

\
\\

\section{Introduction}\label{sec:intro}
Change-point detection and localization is a classical problem in time series analysis, in which we record a series of measurements and wish to determine whether and at what time(s) the underlying generative model has changed. Due to its flexibility, the model of a time series with multiple structural changes has a wide range of applications including econometrics [\cite{bai1998estimating}], epidemiology [\cite{Jiang2021}], stock price analysis [\cite{chen1997testing}], Internet security monitoring  [\cite{peng2004proactively}], and genetics [\cite{castro2018model,Zhao2021}]. 

Change-point detection is mostly studied and well understood in the mean change-point model, where we typically assume we observe a time series $\{y_t\}_{t=1}^n \subset  \mathbb R^p$ such that 
$$y_t = \beta_t^* +\varepsilon_t , \text{ for all }  1 \le t \le n.$$ 
Here $\{\varepsilon_t\}_{t=1}^n$ are independently and identically distributed measurement noise with mean zero and $\{\beta_t^*\}_{t=1}^n $ are the population mean  vectors that change over time in a piecewise constant manner. The important task is to determine whether and where the structural changes of $\{\beta_t^*\}_{t=1}^n$ take place. There is a vast literature of change-point detection in mean for both low and high dimensions, see for example \cite{frick2014multiscale}, \cite{cho2015multiple}, \cite{cho2016change}, \cite{Yau2016} and \cite{wang2018high}. More recently, \cite{pein2017heterogeneous} introduced a method  that can handle  mean and variance changes simultaneously. \cite{cribben2017estimating}, \cite{wang2021optimal} and \cite{Zhao2019}, among others, investigated the mean change-point problem for the dynamic Bernoulli network models. \cite{enikeeva2019high}   studied the optimal  change point detection boundary in the high-dimensional settings. \cite{xie2020sequential} considered online monitoring    change point detection for streaming data in high dimensions. 

However, in some other practical settings , we can only obtain indirect measurements of the (potentially high-dimensional) vectors $\{ \beta_t^*\}_{t=1}^n$. Specifically, in this paper, we consider change-point detection in high-dimensional linear regression. We assume we observe the time series $\{ x_t, y_t\}_{t=1}^n$, where $ \{ x_t\}_{t=1}^n $ are $p$-dimensional covariates, $\{y_t\}_{t=1}^n$ are the univariate responses satisfying $\mathbb E(y_t|x_t) = x_t^\top \beta_t^* \text{ for all } 1\le t \le n $ and $\{\beta_t^*\}_{t=1}^n $ are the unobserved regression coefficients that potentially change over time. We formally summarize the model as follows.

\begin{model}[Change-point model in the regression setting]\label{assume:change-point regression model}
Suppose for $1\le t \le n$, the random covariate $x_t \in \mathbb R^p$ and response $y_t \in \mathbb R $ satisfy 
\begin{align}\label{eq:change-point regression model}
y_t =x_t^\top \beta_t^* +\eps_t,
\end{align}
where the noise $\eps_t \overset{i.i.d.}{\sim} \mathcal{N} (0,\sigma_\varepsilon^2)$ and is independent of the covariate $x_t\overset{i.i.d.}{\sim} \mathcal{N} (0,\Sigma)$\footnote{We assume for convenience that $\eps_t$ and $x_t$ are normally distributed. However our results remain valid as long as $\eps_t$ and $x_t$ are i.i.d. sub-Gaussian random variables.}. In addition, there exist $K\geq 0$ change-points $\{\eta_{k}\}_{k=1}^{K} \subset \{1,\ldots,n-1 \}$ such that 
$$\beta_t^* =\beta_{t'}^* \text{\quad  if \ }  \eta_{k-1}+1 \leq t \le t'\leq  \eta_k, \text{\quad for all \ } k=1,\cdots, K+1,$$
where by convention we define $\eta_0=0$ and $\eta_{K+1}=n$.
\end{model}

\textbf{Notation}: Before we proceed, for clarity of presentation, we first introduce necessary notation used throughout the paper. For two positive sequences $\{a_n\}_{n=1}^\infty$ and $\{b_n\}_{n=1}^\infty$, we write $a_n=O(b_n) $ if there exists $C>0$ such that $\lim\sup_{n\to \infty} a_n/b_n <C $ and write $a_n \asymp b_n $ or $a_n =\Theta(b_n)$ if $a_n =O(b_n)$ and $b_n = O(a_n)$. We write $a_n \succeq b_n$ if $\lim\inf_{n\to \infty} a_n/b_n =\infty.$ Let $\{x_n\}_{n=1}^\infty$ be a sequence of random variables. We write $x_n=O_p(b_n)$ if $x_n/b_n=O_p(1)$ and write $x_n=o_p(b_n)$ if $x_n/b_n=o_p(1)$, where $O_p(1)$ and $o_p(1)$ follow the standard probability notation of big $O$~(stochastic boundedness) and small $o$~(convergence to zero in probability). For a vector $\beta\in\mathbb{R}^p$, denote $\|\beta\|_2=\sqrt{\sum_{i=1}^p{\beta_i^2}}$ as its $l_2$ norm, denote $\|\beta\|_\infty=\max_{1\leq i\leq p}|\beta_i|$ as its $l_\infty$ norm, and denote $\|\beta\|_0=\sum_{i=1}^p\mathbb{I}(\beta_i\neq 0)$ as its $l_0$ norm, where $\beta_i$ denotes the $i$th element of $\beta$ and $\mathbb{I}$ is the indicator function. Given two natural numbers $s<e$, for simplicity, with a slight abuse of notation, we denote $[s,e]:=\{t\in\mathbb{N}|s\leq t\leq e\}$ and denote $(s,e]:=\{t\in\mathbb{N}|s< t\leq e\}$. In other words, $[s,e]$ contains all natural numbers from $s$ to $e$~(inclusive) and $(s,e]$ contains all natural numbers from $s+1$ to $e$~(inclusive). Throughout the paper, we use $c$ and $C$ to denote generic absolute constants independent of $n$ and $p$, and the value of $c$ and $C$ may vary from place to place.

For change-point detection in Model \eqref{assume:change-point regression model}, the key task is to estimate the unknown $\{\eta_{k}\}_{k=1}^{K}$. For any change-point estimator $\{ \widehat \eta_k\}_{k=1}^{K'}$, we say it is consistent if,  with probability approaching 1, $K'=K $ and the sup-norm error satisfies
\begin{align}
\label{eq:sup-norm}
\epsilon :=\max_{1\le k \le K}   \frac{| \widehat \eta_k -\eta_k |}{n} =o_p(1).
\end{align}
for all sufficiently large $n$.  

In the literature, change-point detection for low-dimensional $(p\ll n)$ linear regression models has been extensively studied by many authors including \cite{bai1998estimating}, \cite{qu2007estimating}, and more recently \cite{zhang2015multiple}. Most of the existing works in this setting focus on the case where the number of change-points, $K$, is a fixed constant.

Change-point detection for the high-dimensional linear regression model where $p\gg n$, has also received recent attention. In particular, \cite{lee2016lasso} extended Lasso to the high-dimensional single change-point setting and showed that both the change-point and the regression parameters $\{\beta_t^*\}_{t=1}^n$ can be consistently estimated. Later, \cite{lee2018oracle} extended their results to the high-dimensional quantile regression model. \cite{kaul2018parameter} proposed a highly efficient algorithm for the setting  of exactly one change-point. Both \cite{lee2016lasso} and \cite{kaul2018parameter} showed that in the \emph{single} change-point setting,  the change-point  can be estimated with sup-norm error satisfying $\epsilon= O_p(1/n)$. \cite{zhang2015change} studied the Sparse Group Lasso (SGL) algorithm for the multiple change-points setting. The authors  showed that SGL returns consistent change-point estimators  with $\epsilon=o_p(1)$ when the number of change-points $K$ is bounded. \cite{leonardi2016computationally} showed that, by using a binary  search algorithm, consistent estimation can be achieved with $\epsilon=O_p(1/\sqrt n)$ even when the number of change-points $K$ diverges as $n\to \infty$.

In this paper, we focus on the high-dimensional regime~($p\gg n$) and propose a computationally efficient algorithm that can consistently estimate the unknown multiple change-points at the minimax optimal localization rate $O_p(1/n)$ up to a log factor. To the best of our knowledge, no other method in the literature can achieve this rate for multiple change-points estimation under such setting. We refer to more detailed discussion of our contribution at the end of this section.

We proceed by imposing some mild general assumptions on the high-dimensional regression setting in Model \eqref{assume:change-point regression model} and define key quantities that are used to quantify the localization error rate and requirements on the signal-to-noise ratio~(SNR) of various change-point estimation methods for Model \eqref{assume:change-point regression model}.

\begin{assumption}\label{assume: model assumption}
	\
	\\
	{\bf a.} [Design matrix] There exist absolute positive constants $c_x$ and $C_x$ such that the minimal and maximal eigenvalues of the covariance matrix $\Sigma$ satisfy $\Lambda_{\min}(\Sigma)\ge c_x$ and $\Lambda_{\max} (\Sigma ) \le C_x$.
	\\
	{\bf b.} [Sparse support] There exist a collection of subsets $ \{ S_k  \}_{k=1}^{K+1}  \subset \mathbb  R^p $ such that, for all $k=1,\cdots, K+1$,
	$$\beta^*_{t,j} =0    \text{ if } \eta_{k-1}+1\le t \le \eta_{k+1} \text{ and } j \not \in S_k.$$
	In addition, the size of the support satisfies $\max_{1\le k \le K+1 }|S_k|\le  \mathfrak s$ and there exists an absolute constant $C_\beta$ such that $\max_{1\le t \le n }\| \beta_t^*\|_\infty \le C_\beta<\infty.$  
	 
\end{assumption}

\textbf{Key quantities}: Define $\mathfrak{N}:=\max_{1\le t\le n }\|\beta_t^*\|_2^2$. 
By Assumption \ref{assume: model assumption}{\bf b}, we have $\mathfrak{N}\le C_\beta^2 \s$. Moreover, by Assumption \ref{assume: model assumption}{\bf a}, we have $\text{Var}(y_t)=\beta_t^{*\top}\Sigma \beta_t^*+\sigma_\varepsilon^2 >c_x \|\beta_t^*\|_2^2$, and thus $\N < \max_{1\leq t\leq n} \text{Var}(y_t)/c_x$. For $k=1,\cdots,K+1$, denote $\Delta_k=\eta_{k}-\eta_{k-1}$ as the spacing between two consecutive change-points and define $\Delta=\min_{1\le k \le K+1} \Delta_k$ as the minimum spacing. In addition, for $k=1,\cdots K$, denote $\kappa_k=\left\|\beta^*_{\eta_{k} +1} - \beta_{\eta_k } ^*\right\|_2$ as the $l_2$-norm of the change in regression coefficients and define $\kappa=\min_{1\le k \le K} \kappa_k$ as the minimum change size. Intuitively, the difficulty of change-point detection for Model \ref{assume:change-point regression model} depends on the interplay among $K$, $\kappa$, $\Delta$, $\mathfrak{s}$, the dimension $p$ and the sample size $n$.

We remark that our later theoretical analysis allows the number of change-points $K$, the minimum change size $\kappa$, the minimum spacing $\Delta$, the sparsity $\mathfrak{s}$ and the dimension $p$ to vary with the sample size $n$. To our best knowledge, this is among the most flexible frameworks in the literature.

\textbf{Our contributions}: For change-point estimation in the high-dimensional regression model, we propose a novel two-stage detection procedure named Variance Projected Wild Binary Segmentation~(VPWBS). Given the observations $\{x_t, y_t\}_{t=1}^n$, in Stage 1, VPWBS estimates the regression coefficients $\{\beta_t^*\}_{t=1}^n$ using a group Lasso based local screening algorithm carefully tailored for the high-dimensional regression change-point setting. In Stage 2, via a novel projection step, VPWBS projects the high-dimensional regression data $\{x_t, y_t\}_{t=1}^n$ into a one-dimensional time series, where the optimal projection direction is derived from the estimated $\{\widehat \beta_t\}_{t=1}^n$ in Stage 1. Subsequently, VPWBS achieves change-point estimation by performing mean change detection via CUSUM statistics on the resulting one-dimensional time series.

Our theoretical analysis shows that VPWBS can achieve consistent estimation even when the number of change-points $K$ diverges as $n\to \infty$. Furthermore, the sup-norm error $\epsilon$ (defined in \eqref{eq:sup-norm}) of the VPWBS change-point estimator is, up to a log factor, of order $O_p(1/n)$, which is the known minimax optimal rate. To the best of our knowledge, this is a significant improvement for \textit{multiple} change-point estimation in the high-dimensional regression setting, as the aforementioned existing literature can only achieve $\epsilon= O_p(1/\sqrt {n} )$ at best. A key step of VPWBS is the estimation of an optimal projection direction. In the theoretical analysis, we establish error bounds on the estimated high-dimensional projection direction, which may be of independent interest. VPWBS admits a reasonable computational cost of order $O(n(\log(n))^2 \cdot \text{GroupLasso}(n,p))$, which enables its implementation in the high-dimensional regression setting. Here $\text{GroupLasso}(n,p)$ denotes the computational cost of the group Lasso for a $p$-dimensional regression with $n$ samples. Similar definition applies to $\text{Lasso}(n,p)$. We summarize the localization error bound and computational cost of VPWBS and two other state-of-the-art methods in \Cref{table for comparison} and refer more detailed discussions to Sections \ref{sec:vpwbs} and \ref{sec:sim}.
 
\begin{table}[H] 
	\renewcommand{\arraystretch}{1.5}
	\begin{center}
		\begin{tabular}{l | c | c | c | c }  
			& Localization  Error Bound $\epsilon$ &  SNR Condition &   Computational Complexity  \\
			\hline
			VPWBS & $O_p\big (    \mathfrak{N }\log(n)   n^{-1  }   \big  )$   &  $\Delta \kappa^2 \succeq     \s     \log(pn) $   & $O(n(\log(n))^2 \cdot \text{GroupLasso } ( n,p )  ) $ 
			\\
			\hline
			EBSA  &  $O_p\big (     \s     \log(p)  n^{-1/2 }   \big  )$       &    $\Delta \kappa^2 \succeq    \mathfrak{N }   \s     \log(p) $&  $O(n\log(n) \cdot \text{Lasso } ( n,p )  ) $ \\
			\hline 
			SGL      & $ o_p(1)$   &    $\Delta \kappa^2  \asymp n $  & $O(\text{Lasso} (n,np))$\\
			\hline 
		\end{tabular}
		\caption{SGL~\citep{zhang2015multiple} and EBSA~\citep{leonardi2016computationally} are two state-of-the-art methods developed for change-point estimation in high-dimensional regression. Recall $\mathfrak{N}\le C_\beta^2 \s$ and we refer the detailed definition of notation $n,p,\Delta,\kappa,\s,\mathfrak{N}$ to Assumption \ref{assume: model assumption} and the discussion on key quantities.}
		\label{table for comparison}
	\end{center} 
\end{table}

The rest of the paper is organized as follows. In Section \ref{sec:general_framework}, we introduce the projection based change-point estimation framework and a group Lasso based local screening algorithm for the estimation of the optimal projection direction. Building upon wild binary segmentation, Section \ref{sec:vpwbs} proposes the VPWBS for multiple change-point estimation in high-dimensional regression and further establishes its optimal theoretical properties. Extensive numerical experiments are conducted in Section \ref{sec:sim} to demonstrate the promising performance of VPWBS when compared with state-of-the-art methods in the literature. Section \ref{sec:discussion} concludes with a discussion. Technical proofs can be found in the supplementary material.


\section{A General Framework and Group Lasso Based Screening}\label{sec:general_framework}
In this section, we introduce the general framework of the proposed change-point estimation procedure for the high-dimensional regression problem in Model \eqref{assume:change-point regression model}. Specifically, Section \ref{subsec:projection} discusses the essential idea of a projection based change-point detection framework and Section \ref{subsec:group_lasso} proposes a group Lasso based screening algorithm for estimating the unknown projection direction.

\subsection{A projection based change-point estimation framework}\label{subsec:projection}
To ease presentation, we start the discussion with the problem of single change-point estimation. Specifically, given a sample of high-dimensional regression $\{x_t,y_t\}_{t=1}^n$ with $y_t=x_t^\top \beta_t^*+\varepsilon_t$, assume there is a single change-point at an unknown time point $\eta$ such that
$$\beta_t^*=\beta^{(1)} \text{ for } 1\leq t\leq \eta \text{ and } \beta_t^*=\beta^{(2)} \text{ for } \eta+1\leq t\leq n.$$ To detect the existence of $\eta$ and further estimate its location, we need to measure and test the difference between the unknown regression coefficients $\beta^{(1)}$ and $\beta^{(2)}$.

For two regression coefficients $\beta^{(1)}$ and $\beta^{(2)}$, it is natural to directly measure their difference via the $l_2$-norm $\|\beta^{(1)}-\beta^{(2)}\|_2^2.$ However, under the regression context, an arguably more relevant alternative is ${(\beta^{(1)}-\beta^{(2)})^\top \Sigma (\beta^{(1)}-\beta^{(2)})}$, which equals to ${\text{Var}(x_t^\top(\beta^{(1)}-\beta^{(2)}))}$ as $\text{Var}(x_t)=\Sigma$. Note that under Assumption \ref{assume: model assumption}\textbf{a}, we have that
$$c_x\|\beta^{(1)}-\beta^{(2)}\|_2^2 \leq {(\beta^{(1)}-\beta^{(2)})^\top \Sigma (\beta^{(1)}-\beta^{(2)})} \leq C_x \|\beta^{(1)}-\beta^{(2)}\|_2^2.$$
Thus, in terms of theoretical magnitude, $\|\beta^{(1)}-\beta^{(2)}\|_2^2$ and ${(\beta^{(1)}-\beta^{(2)})^\top \Sigma (\beta^{(1)}-\beta^{(2)})}$ are the same and both can capture the change in the regression coefficient. However, compared to $\|\beta^{(1)}-\beta^{(2)}\|_2^2$, the quantity ${(\beta^{(1)}-\beta^{(2)})^\top \Sigma (\beta^{(1)}-\beta^{(2)})}$ further incorporates the covariance structure $\Sigma$ of the covariates and thus can better reflect the difference between two regression models $y=x^\top \beta^{(1)}+\varepsilon$ and $y=x^\top \beta^{(2)}+\varepsilon$. We therefore prefer ${(\beta^{(1)}-\beta^{(2)})^\top \Sigma (\beta^{(1)}-\beta^{(2)})}$ for change-point estimation. We remark that ${(\beta^{(1)}-\beta^{(2)})^\top \Sigma (\beta^{(1)}-\beta^{(2)})}$ is closely related to the explained variance in the regression literature, see for example \cite{Cai2020}.

For any $1\leq m\leq n-1$, define $\beta^{(1)}_m=\sum_{t=1}^{m}\beta_t^*/m$ and $\beta^{(2)}_m=\sum_{t=m+1}^{n}\beta_t^*/(n-m)$. Note that $\beta^{(1)}_m$ and $\beta^{(2)}_m$ are the unique minimizer of the population squared loss function $\mathbb{E}(\sum_{t=1}^{m}(y_t-x_t^\top \beta)^2)$ and $\mathbb{E}(\sum_{t=m+1}^{n}(y_t-x_t^\top \beta)^2)$, respectively. As a function of $m$, ${(\beta_m^{(1)}-\beta_m^{(2)})^\top \Sigma (\beta_m^{(1)}-\beta_m^{(2)})}$ achieves its maximum at the true change-point $m=\eta$ due to the fact that  $\beta_m^{(1)}-\beta_m^{(2)}= \min(\frac{\eta}{m},\frac{n-\eta}{n-m})(\beta^{(1)}-\beta^{(2)})$. Thus, the sample estimate of ${(\beta_m^{(1)}-\beta_m^{(2)})^\top \Sigma (\beta_m^{(1)}-\beta_m^{(2)})}$ can be valuable for the detection and estimation of $\eta.$

Given a time point $m$, to estimate ${(\beta_m^{(1)}-\beta_m^{(2)})^\top \Sigma (\beta_m^{(1)}-\beta_m^{(2)})}$, a natural choice is the plug-in estimator. Specifically, via a penalized M-estimator, we can obtain $\widehat{\beta}^{(1)}_m$ from $\{x_t,y_t\}_{t=1}^{m}$ and $\widehat{\beta}^{(2)}_m$ from $\{x_t,y_t\}_{t=m+1}^{n}$. Combined with a covariance matrix estimator $\widehat{\Sigma}$, the plug-in estimator takes the form $(\widehat{\beta}^{(1)}_m-\widehat{\beta}^{(2)}_m)^\top \widehat{\Sigma}(\widehat{\beta}^{(1)}_m-\widehat{\beta}^{(2)}_m)$. This in some sense resembles the classical Wald-type statistics used in the change-point literature, see for example \cite{Davis1995} and \cite{Huskova2007}. However, the plug-in estimator requires the estimation of $\Sigma$. Without strong structural assumptions on $\Sigma$, this is known to be a difficult task in high dimensions.

To bypass this difficulty, we slightly alter the estimation target and propose an alternative estimator via projection. Specifically, given a $p$-dimensional unit vector $u$ with $\|u\|_2=1$, we define the \textit{one-dimensional} variance-projected time series $\{z_t(u)\}_{t=1}^n$ as
$$z_t(u)=u^\top x_ty_t, \text{ for }t=1,\cdots,n.$$
A key observation is that $\{z_t(u)\}_{t=1}^n$ has a single change-point in mean at time point $\eta$ as long as $u^\top \Sigma (\beta^{(1)}-\beta^{(2)})\neq 0$. Importantly, if $u=(\beta^{(1)}-\beta^{(2)})/\|\beta^{(1)}-\beta^{(2)}\|_2$, we have that
\begin{align}\label{eq:projection_mean}
    &\mathbb{E}\left(\frac{1}{m}\sum_{t=1}^{m}z_t(u)-\frac{1}{n-m}\sum_{t=m+1}^{n}z_t(u)\right)\nonumber=u^\top \Sigma (\beta^{(1)}_m-\beta^{(2)}_m)\\=&\min(\frac{\eta}{m},\frac{n-\eta}{n-m})\frac{(\beta^{(1)}-\beta^{(2)})^\top \Sigma (\beta^{(1)}-\beta^{(2)})}{\|\beta^{(1)}-\beta^{(2)}\|_2},
\end{align}
which is proportional to the key quantity $(\beta^{(1)}-\beta^{(2)})^\top \Sigma (\beta^{(1)}-\beta^{(2)})$ and also achieves its maximum at $m=\eta.$ Note that we further have
$$\eqref{eq:projection_mean}\geq \min(\frac{\eta}{m},\frac{n-\eta}{n-m}) c_x\|\beta^{(1)}-\beta^{(2)}\|_2.$$
Thus, the projection direction $u=(\beta^{(1)}-\beta^{(2)})/\|\beta^{(1)}-\beta^{(2)}\|_2$ is optimal in the sense that it preserves the \textit{original} change size $\|\beta^{(1)}-\beta^{(2)}\|_2$ of the regression coefficients. Therefore, if the projection direction $u$ is reasonably aligned with $\beta^{(1)}-\beta^{(2)}$, we can efficiently detect and estimate the change-point $\eta$ by performing change-point estimation in mean on the univariate time series $\{z_t(u)\}_{t=1}^n$. To estimate the optimal projection direction, in Section \ref{subsec:group_lasso}, we propose a group Lasso based local screening~(LGS) algorithm which provides an estimated $\widehat{\beta}^{(1)}-\widehat{\beta}^{(2)}$.

Note that the above projection framework loses its intuition and becomes less effective when $\{x_t,y_t\}_{t=1}^n$ contains multiple change-points. To tackle this issue, in Section \ref{sec:vpwbs}, we further combine the projection idea with the wild binary segmentation~(WBS) in \cite{fryzlewicz2014wild} and propose a multiple change-point detection algorithm named variance-projected WBS~(VPWBS). Roughly speaking, the strategy is to perform the projection based change-point detection for $\{x_t,y_t\}_{t=1}^n$ on many randomly generated intervals $\{(a_m,b_m]\}_{m=1}^M$ with $1\leq a_m+1<b_m\leq n$, instead of focusing on the whole sample on $(0,n]$. The hope is that for a sufficiently large $M$, some random intervals will contain only one change-point and the projection based detection method will succeed.

\textbf{An illustrative example}: To facilitate understanding, we provide an illustrative example of how VPWBS works in practice. Specifically, we generate the data $\{x_t,y_t\}_{t=1}^n$ according to simulation setting (i) in Section \ref{subsec:simu_results}, where we have $n=300, p=100$ and there are two change-points of $\{\beta_t^*\}_{t=1}^{n}$ at $\eta_1=100$ and $\eta_2=200$ with change size $\kappa= 1.6\sqrt {40}$. For illustration, we focus on one of the randomly generated intervals $(104, 290]$, which contains a single change-point at $\eta_2=200$. Figure \ref{figure:projected}(a)-(b) plots the subsample observations $\{x_t,y_t\}_{t=105}^{290}$, where no clear pattern of changes can be seen. Based on the above discussion, the optimal projection direction is $$u^*=(\beta^*_{\eta_2}-\beta^*_{\eta_2+1})/\|\beta^*_{\eta_2}-\beta^*_{\eta_2+1}\|_2=(\underbrace{1, -1,1,-1\ldots, -1}_{10}, \underbrace{0, \ldots, 0}_{90})/\sqrt{10}.$$ Figure \ref{figure:projected}(c) plots the projected univariate time series $\{z_t(u^*)=u^{*\top}x_ty_t\}_{t=105}^{290}$ and its one-dimensional CUSUM statistics~(see definition in \eqref{eq: 1D cusum} later). Note that there is a clear pattern of mean change for $\{z_t(u^*)\}_{t=105}^{290}$ around the true change-point $\eta_2=200$ and the CUSUM statistics is indeed maximized at $t=200$. Figure \ref{figure:projected}(d) plots the projected univariate time series $\{z_t(\widehat{u})=\widehat{u}^{\top}x_ty_t\}_{t=105}^{290}$ and its CUSUM statistics, where $\widehat{u}$ is estimated by the LGS algorithm in Section \ref{subsec:group_lasso} using $\{x_t,y_t\}_{t=105}^{290}$  As can be seen, Figure \ref{figure:projected}(d) closely resembles Figure \ref{figure:projected}(c) and thus confirms the success of the proposed projection based change-point estimation framework.

\begin{figure}[h]
	\begin{center}
		\includegraphics[scale=0.45]{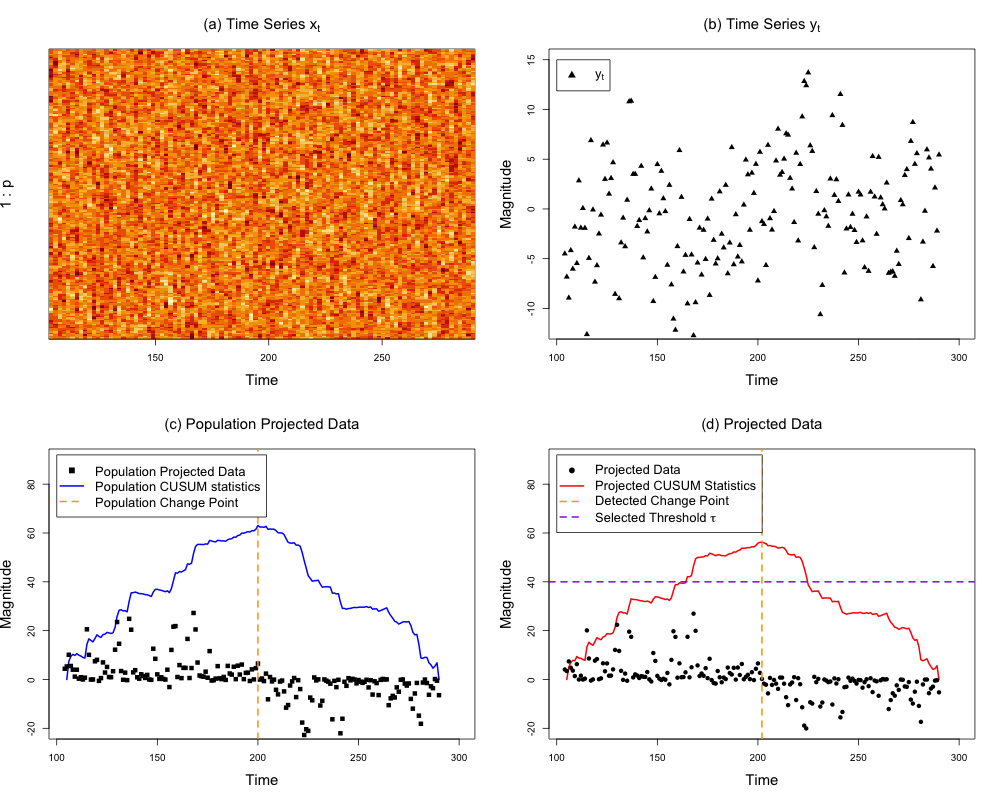} 
		\caption{Plots of (a) $\{x_t\}_{t=105}^{290}$ (b) $\{y_t\}_{t=105}^{290}$ (c) $\{z_t(u^*)\}_{t=105}^{290}$ and its CUSUM statistics (d) $\{z_t(\widehat{u})\}_{t=105}^{290}$ and its CUSUM statistics.
		} \label{figure:projected}
	\end{center}
\end{figure}

\begin{remark}
	Given the estimated $p$-dimensional vector $\widehat{\beta}^{(1)}-\widehat{\beta}^{(2)}$, an intuitive and tempting alternative option is to detect change-points directly based on $\|\widehat{\beta}^{(1)}-\widehat{\beta}^{(2)}\|_2^2$. However, we remark that the extra projection step in our proposed framework helps further turn (and simplify) the $p$-dimensional problem into one-dimensional change-point detection in mean. This projection step acts as a refinement and is essential for the proposed method to achieve the minimax optimal rate~(up to a log factor). In comparison, estimation error may accumulate along the $p$ coordinates for $\|\widehat{\beta}^{(1)}-\widehat{\beta}^{(2)}\|_2^2$, making its theoretical analysis much more challenging. See \cite{wang2018high} for a similar observation in change-point detection for high-dimensional mean.
\end{remark}

\subsection{Local Group Lasso Screening~(LGS)}\label{subsec:group_lasso}
In this section, we propose a local group Lasso based screening~(LGS) algorithm for estimating the optimal projection direction given the observed high-dimensional regression $\{x_t,y_t\}_{t=1}^n$. 

Specifically, denote $1\leq s+1<e\leq n$ as the subsample index, LGS performs a variant of the group Lasso on the subsample $\{x_t,y_t\}_{t=s+1}^e$ and computes
\begin{align}  \label{eq:group lasso algorithm} 
	\begin{split}
		\left(\widehat{\alpha}_1, \widehat{\alpha }_2, \widehat{ \nu } \right) \leftarrow 
		\argmin_{\substack{\nu \in [s'+1, e'-1], \\ \alpha_1, \alpha_2 \in \mathbb{R}^{p}}} \Bigg\{&\sum_{t=s+1}^{\nu} (y_{t} - x_t ^{\top} \alpha_1 \bigr )^2  
		+  \sum_{t=\nu+1}^{e}(y_t-x_t^\top \alpha_2)^2  \\
		+& \lambda \sum_{i=1}^p \sqrt{(\nu-s)(\alpha_{1,i})^2 + (e-\nu)(\alpha_{2,i})^2}\Bigg\},
	\end{split} 
\end{align}
where $s'$ and $e'$ serve as boundary trimming parameters with $s+1\leq s'+1< e'\leq e$, and $\lambda$ is the tuning parameter for the group penalty. In the following, for convenience, we set $s'=s+\lfloor (e-s)/10\rfloor$ and $e'= e-\lfloor(e-s)/10\rfloor$, and summarize the detailed implementation of LGS in Algorithm \ref{alg:LGS}.

\begin{algorithm}[H]
	\begin{algorithmic}
		\INPUT data $\{x_t, y_t\}_{t=1}^{ n }$, subsample index $(s,e]$, tuning parameter $\lambda$.
		\State Set $s'=s+\lfloor (e-s)/10\rfloor$ and $e'= e-\lfloor(e-s)/10\rfloor$ and compute
		\begin{align*}  
			\begin{split}
				\left(\widehat{\alpha}_1, \widehat{\alpha }_2, \widehat{ \nu } \right) \leftarrow 
				\argmin_{\substack{\nu \in [s'+1, e'-1], \\ \alpha_1, \alpha_2 \in \mathbb{R}^{p}}} \Bigg\{&\sum_{t=s+1}^{\nu} (y_{t} - 	x_t ^{\top} \alpha_1 \bigr )^2  
				+  \sum_{t=\nu+1}^{e}(y_t-x_t^\top \alpha_2)^2  \\
				+& \lambda \sum_{i=1}^p \sqrt{(\nu-s)(\alpha_{1,i})^2 + (e-\nu)(\alpha_{2,i})^2}\Bigg\}.
			\end{split} 
		\end{align*}
		\OUTPUT $\{\widehat{\alpha}_1, \widehat{\alpha }_2, \widehat{\nu } \}$.
		\caption{Local group Lasso based Screening. LGS $(\{x_t, y_t\}_{t=1}^{n}, (s,e],\lambda) $.  } 
		\label{alg:LGS} 
	\end{algorithmic}
\end{algorithm}

The proposed LGS algorithm in \eqref{eq:group lasso algorithm} is different from the classical group Lasso or sparse group Lasso, as LGS explicitly targets the single change-point alternative in its formulation by incorporating two separate regression coefficients $\alpha_1$ and $\alpha_2.$ Intuitively, when the subsample is sufficiently large w.r.t.\ the signal-to-noise ratio~(SNR) and contains only one change-point $\eta$, the output $\widehat{\alpha}_1-\widehat{\alpha}_2$ of LGS can estimate the optimal projection direction $\beta_{\eta}^*-\beta_{\eta+1}^*$ accurately. Assumption \ref{assume: LGS assumption} formalizes this intuition and Theorem \ref{corollary:property of projection} further establishes the approximation quality of $\widehat{\alpha}_1-\widehat{\alpha}_2$. Recall the definition in Section \ref{sec:intro} that $\kappa=\min_{1\le k \le K}\left\|\beta^*_{\eta_{k} +1} - \beta_{\eta_k } ^*\right\|_2$ denotes the minimum change size and $\Delta=\min_{1\le k \le K+1} (\eta_{k}-\eta_{k-1})$ denotes the minimum spacing between change-points.

\begin{assumption}\label{assume: LGS assumption}
	\
	\\
	{\bf a.} There exists an absolute constant $C_\kappa$ such that $ \left\|\beta^*_{\eta_{k} +1} - \beta_{\eta_k } ^*\right\|_2 \le C_\kappa$ for all $k=1,\cdots, K.$
	\\
	{\bf b.} [SNR] We have $\Delta \kappa^2  \ge  C_{snr}\s \log(pn)$ where $C_{snr}=C_{snr}(n)$ is a diverging sequence as $n\to \infty.$
\end{assumption}
\noindent Assumption \ref{assume: LGS assumption}{\bf a} is a technical condition needed in the proof, which is also used in \cite{lee2016lasso} and \cite{kaul2018parameter}. Assumption \ref{assume: LGS assumption}{\bf b} implies that $\Delta \ge  C_{snr}\kappa^{-2}\s \log(pn) \ge C_{snr}C_\kappa^{-2} \s \log(pn)$, which is the standard SNR condition in the Lasso literature. Note that we require $C_{snr}\to\infty$ as $n\to \infty$, but the divergence rate can be arbitrarily slow.

\begin{theorem} \label{corollary:property of projection}
	Suppose Assumptions \ref{assume: model assumption}-\ref{assume: LGS assumption} hold  and  $\lambda=C_\lambda\sqrt {\log(pn)}$ for some sufficiently large constant $C_\lambda$. Let $\left(\widehat{\alpha}_1, \widehat{\alpha }_2, \widehat{ \nu } \right)$ be the output of LGS $(\{x_t, y_t\}_{t=1}^{n}, (s,e], \lambda)$. Suppose that $(s, e]$ satisfies $e-s\geq \Delta/2$ and contains exactly one change-point $\eta$ such that
	$$\min \{ \eta-s, e-\eta  \} \ge \frac{e-s}{10}.$$
	Then with probability at least $1-2(pn)^{-4} $, it holds that 
	$$\big  \|   ( \widehat \alpha _1 - \widehat \alpha _2)  - (\beta^*_\eta -\beta^*_{\eta+1 })  \big\|_2 \le \frac{c_x   }{32 C_x } \|  \beta^*_\eta -\beta^*_{\eta+1 }  \| _2. $$
\end{theorem}

Theorem \ref{corollary:property of projection} states that when the subsample $\{x_t,y_t\}_{t=s+1}^e$ contains only one change-point $\eta$ and has sufficient number of observations, the proposed LGS algorithm can accurately estimate the optimal projection direction $\beta_{\eta}^*-\beta_{\eta+1}^*$, which serves as the foundation for the later theoretical analysis of the projection based framework.  
\section{Variance-Projected Wild Binary Segmentation}\label{sec:vpwbs}
In this section, we formalize the discussion in Section \ref{sec:general_framework} and present the variance-projected wild binary segmentation~(VPWBS) algorithm for multiple change-point estimation in high-dimensional linear regression of Model \eqref{assume:change-point regression model}.

Note that the LGS algorithm and the projection framework in Section \ref{sec:general_framework} are discussed under the single change-point scenario. To further extend to multiple change-point estimation, VPWBS employs the mechanism of wild binary segmentation in \cite{fryzlewicz2014wild}, where the essential idea is to perform single change-point estimation on $M$ randomly generated intervals $\left\{(a_m,b_m]\right\}_{m=1}^M$ where $1\leq a_m+1<b_m\leq n$. For a sufficiently large $M$, with high probability, for every true change-point in $\{\eta_k\}_{k=1}^K$, there exists at least one random interval $(a_m,b_m]$ such that $\eta_k$ is the only change-point contained in $(a_m,b_m]$. More specifically, the good event
\begin{align*}
	\mathcal{M}=\bigcap_{k=1}^K \{a_m\in \mathcal{S}_k, b_m\in\mathcal{E}_k, \text{ for some } m \in \{1,2,\cdots,M\} \}
\end{align*}
will hold with high probability, where $\mathcal{S}_k=(\eta_k-3\Delta/4,\eta_k-\Delta/2]$ and $\mathcal{E}_k=(\eta_k+\Delta/2,\eta_k+3\Delta/4]$, for $k=1,2,\cdots, K.$ It is easy to see that if $a_m\in \mathcal{S}_k$ and $b_m\in \mathcal{E}_k$, we have $(a_m,b_m]$ only contains a single change-point $\eta_k$, as by definition the minimum spacing between two consecutive change-points is $\Delta.$ Theorem \ref{theorem:vpcusum} later provides a rigorous bound for the probability that event $\mathcal{M}$ holds.


Another issue that needs to be addressed is that for the projection idea in \eqref{eq:projection_mean} to be theoretically valid, the projection direction $u$ is required to be independent from the observations $\{x_t,y_t\}_{t=1}^n.$ To tackle this issue, we use sample splitting, a commonly used technique in high-dimensional statistics, see for example \cite{wang2018high}, \cite{wang2021optimal} and \cite{Zou2020}. Without loss of generality, we assume the original sample $\{x_t,y_t\}_{t=1}^{2n}$ is of length $2n$~(i.e.\ even) and we estimate the projection direction using LGS on the oddly-indexed observations $\{x_t^{(1)},y_t^{(1)}\}_{t=1}^n$ and perform change-point estimation on the projected univariate  series based on the evenly-indexed observations $\{x_t^{(2)},y_t^{(2)}\}_{t=1}^n$, where
$$(x_{t}^{(1)},y_t^{(1)})=(x_{2t-1},y_{2t-1}) \text{ and } (x_{t}^{(2)},y_t^{(2)})=(x_{2t},y_{2t}) \text{ for } t=1,\cdots,n.$$

To summarize, VPWBS implements the following two-stage procedure. In the first stage, given $M$ random intervals $\left\{(a_m,b_m]\right\}_{m=1}^M$, the LGS in Algorithm \ref{alg:LGS} is implemented on $\{x_t^{(1)},y_t^{(1)}\}_{t=1}^n$ for each of the $M$ subsamples indexed by $(a_m,b_m]$ and returns $M$ projection directions $\{{u}_m\}_{m=1}^M$. In the second stage, based on $\{{u}_m\}_{m=1}^M$ and $\{x_t^{(2)},y_t^{(2)}\}_{t=1}^n$, we conduct mean change-point detection on the projected univariate time series $\{z_t({u}_m)={u}_m^\top x_t^{(2)}y_t^{(2)}\}_{t=a_m+1}^{b_m}$ for $m=1,\cdots, M$ via the classical CUSUM statistics. For a univariate series $\{z_t(u_m)\}_{t=a_m+1}^{b_m}$ and $a_m\leq s_m<\nu <e_m\leq b_m$, the CUSUM statistics computed on $\{z_t(u_m)\}_{t=s_m+1}^{e_m}$ is defined as
\begin{align}\label{eq: 1D cusum}
	\widetilde Z^{s_m,e_m}_{\nu}(u_m) = \sqrt\frac{e_m-\nu}{(e_m-s_m)(\nu-s_m)} \sum_{t=s_m+1}^\nu z_t(u_m) - \sqrt\frac{\nu-s_m}{(e_m-s_m)(e_m-\nu)}\sum_{t=\nu+1}^{e_m} z_t(u_m).
\end{align}

We summarize the detailed description of VPWBS in Algorithm \ref{alg:VPWBS}. In total, there are four tuning parameters $(M, \lambda, \tau, \zeta)$ of the algorithm, where $M$ is the number of random intervals, $\lambda$ regulates the group Lasso penalty, $\tau$ is the threshold level of the maximum CUSUM statistics and $2\zeta$ is the minimum length required for a subsample $(s_m,e_m]$ to be considered for change-point detection. Theorem \ref{theorem:vpcusum} establishes the consistency and localization rate of VPWBS and gives the theoretical orders required for the tuning parameters $(M, \lambda, \tau, \zeta)$. We refer more details to the discussion after Theorem \ref{theorem:vpcusum}.

\begin{algorithm}[ht]
	\begin{algorithmic}
		\INPUT 1st sample $\{x_t^{(1)},y_t^{(1)}\}_{t=1}^n$, 2nd sample $\{x_t^{(2)},y_t^{(2)}\}_{t=1}^n$, random intervals $\{ (a_m,b _m]\}_{m=1}^M$, tuning parameters  $\lambda>0 $, $\tau>0$,  $ \zeta >0 $.
		
		\vspace{0.2cm}
		\noindent \hspace{-0.4cm} \textbf{Initialize} the set of estimated change-points as $\textbf{S}=\varnothing$ and set $(s,e]=(0,n]$.
		\vspace{0.1cm}
		
		\noindent \hspace{-0.4cm} \textbf{Stage 1}: LGS and projection
		\For{$m = 1, \ldots, M$}  
		\State compute 
		$ \{\widehat{\alpha}_1 ^m , \widehat{\alpha }_2^m \}  \leftarrow \text{ LGS } (\{x_t^{(1)},y_t^{(1)}\}_{t=1}^n, (a_m,b _m], \lambda) .$
		\State set the projection direction: $u_m  \leftarrow {(\widehat{\alpha }_2^m -\widehat{\alpha}_1^m)}/{\| \widehat{\alpha }_2^m -\widehat{\alpha}_1 ^m \|_2}.$
		\State set the projected univariate series: $z_{t} (u_m) \leftarrow u_m^\top x_t^{(2)} y_t^{(2)} \text{ for } t\in (a_m,b_m].$ 
		\EndFor
		
		\vspace{0.1cm}
		\noindent \hspace{-0.4cm} \textbf{Stage 2}: WBS($(s,e], \{ (a_m,b _m]\}_{m=1}^M, \tau ,\zeta$ )
		\For{$m = 1, \ldots, M$}  
		\State $(s_m , e_m] \leftarrow  ( s , e ]\cap  (a _m,  b _m]$
		\If{$e_m - s_m \geq 2\zeta$}
		\State $D_{m} \leftarrow \arg\max_{s_m +\zeta \leq t \leq e_m -\zeta}  | \widetilde Z^{s_m,e_m}_{t} ( u_m )|$  \Comment{Recall \Cref{eq: 1D cusum}}
		\State $A_m \leftarrow \max_{s_m +\zeta   \leq t \leq e_m -\zeta}  | \widetilde Z^{s_m,e_m}_{t} (u_m )| $ 
		\Else 
		\State $A_m \leftarrow-1$	
		\EndIf
		\EndFor	
		\State $m^* \leftarrow \arg\max_{m = 1, \ldots, M} A_{m}$
		\If{$A_{m^*} > \tau$}
		\State add $D_{m^*}$ to the set $\textbf{S}$
		\State  WBS $(   (s, D_{m^* }] ,  \{ (a_m,b _m]\}_{m=1}^M , \tau ,\zeta  )$
		\State  WBS $(  (D_{m^* } ,  e ] ,  \{ (a_m,b _m]\}_{m=1}^M , \tau, \zeta  )$
		\EndIf
		
		\vspace{0.2cm}
		\OUTPUT The set of estimated change-points $\textbf{S}$.
		\caption{Variance-Projected Wild Binary Segmentation.  
			VPWBS $(\{(a_m,b _m]\}_{m=1}^M , \lambda, \tau, \zeta)$} 
		\label{alg:VPWBS}
	\end{algorithmic}
\end{algorithm}

We remark that the sample splitting step of VPWBS in Algorithm \ref{alg:VPWBS} is mainly needed for establishing its theoretical validity in Theorem \ref{theorem:vpcusum}. In practice, we find that VPWBS is often more efficient without sample splitting. In other words, we can set both $\{x_t^{(1)},y_t^{(1)}\}_{t=1}^n$ and $\{x_t^{(2)},y_t^{(2)}\}_{t=1}^n$ as the original sample in Algorithm \ref{alg:VPWBS}. See \cite{wang2018high} for similar phenomenon in high-dimensional mean change-point estimation. Recall from Section \ref{sec:intro} that $\mathfrak{N}:=\max_{1\le t\le n }\|\beta_t^*\|_2^2$ and by assumption \ref{assume: model assumption}, we have $\mathfrak{N}\le C_\beta^2 \s$ and $\N < \max_{1\leq t\leq n} \text{Var}(y_t)/c_x.$

\begin{theorem} \label{theorem:vpcusum}
	Suppose Assumptions \ref{assume: model assumption}-\ref{assume: LGS assumption} hold. Let $\{(a_m, b_m] \}_{m=1}^M $ be a collection of intervals whose end points are drawn independently and uniformly from $\{1,\ldots, n\}$ and that $\max_{1\le m \le M} (b_m -a_m)\le C_R \Delta$ for some absolute constant $C_R>0$.  
	\\
	\indent Let $\{\widehat \eta_k\}_{k=1}^{\widehat K}$ be the estimated change-points by VPWBS with data $\{x_t^{(1)},y_t^{(1)}\}_{t=1}^n$, $\{x_t^{(2)},y_t^{(2)}\}_{t=1}^n$, random intervals $\{ (a_m,b _m]\}_{m=1}^M$, and tuning parameters $\lambda>0 $, $\tau>0$,  $ \zeta >0 $, where 
	$$\lambda =  C_\lambda \sqrt {\log(pn)}, \quad \tau = C_\tau  \sqrt { ( \N +1) \log(n)}, \quad \text{and} \quad \zeta = C_\zeta  (\N+1)  \log(n)$$
	for sufficiently large constants $ C_\lambda , C_\tau $ and $ C_\zeta$. Then there exists an absolute constant $C  $ such that 
	\begin{align}
		\mathbb{P}\Bigl\{  \widehat K =K ; \ |\eta_k-\widehat \eta_k| \le & 
		\frac{C (\N+1)  \log(n)}{\kappa^2_k }  \text{ for all  } 1\le k \le K    \Bigr\}  \nonumber\\
		\geq &   1-  n^{-2}  - \exp \bigg( \log\big( \frac{n}{\Delta} \big) - \frac{M\Delta^2}{16 n^2}\bigg) .  \label{eq:localization bound 2}
	\end{align}	
\end{theorem}
Theorem \ref{theorem:vpcusum} establishes the consistency of VPWBS and further provides the localization error rate. Note that since $\mathfrak{N}\le C_\beta^2 \s$ and $\N < \max_{1\leq t\leq n} \text{Var}(y_t)/c_x$, we have $\mathfrak{N}=O(1)$ if the sparsity level $\s$ is a constant or the maximum variance of the response $y_t$ is upper bounded, which is a rather mild condition. In such case, the localization error bound in Theorem \ref{theorem:vpcusum} further implies
$$  \epsilon  = \max_{1 \le k \le K }  \frac{\ |\eta_k-\widehat \eta_k| }{ n } \le  \max_{1 \le k \le K }  C\frac{\log(n)}{ n\kappa^2_k }, $$
where $\epsilon$ is the localization error defined in \eqref{eq:sup-norm} and  $\kappa_k=\|\beta_{\eta_k }-\beta_{\eta_k+1} \|_2$ is the change size at $\eta_k$. Up to a log factor, this matches the well-known minimax optimal rate for change-point estimation, see \cite{wang2018univariate} and references therein.

Theorem \ref{theorem:vpcusum} requires $\max_{1\le m \le M} (b_m -a_m)\le C_R \Delta$, which essentially implies that the random intervals cannot contain too many change-points. See similar assumptions in \cite{kaul2018parameter}. Note that if $\Delta \asymp n$, the assumption becomes minimal as we can simply set $C_R \Delta=n.$ We remark that Theorem \ref{theorem:vpcusum} still holds without the assumption $\max_{1\le m \le M} (b_m -a_m)\le C_R \Delta$, however, the localization error rate in \eqref{eq:localization bound 2} will be inflated to $({n}/{\Delta})^2 \cdot ({C (\N+1)  \log(n)}/{\kappa^2_k })$ by a factor $({n}/{\Delta})^2$. This is a phenomenon commonly seen in the high-dimensional change-point literature, see for example \cite{wang2018high}, \cite{wang2021optimal} and \cite{Li2021}.

\textbf{Discussion on tuning parameters $(M,\lambda, \tau, \zeta)$}: By the probability bound \eqref{eq:localization bound 2} in \Cref{theorem:vpcusum}, for the consistency of VPWBS, it is necessary to choose the number of random intervals $M \succeq  { n^2 {\log(n)}}/{\Delta^2}$. In particular, suppose that $\Delta \asymp n$~(i.e. there are finite number of change-points), it suffices to choose $M = (\log(n))^2 $. The tuning parameter $\lambda$ is needed in the LGS algorithm and assumes the standard order $C_\lambda \sqrt{\log(pn)}$ of the group Lasso penalty in the literature~(see also Theorem \ref{corollary:property of projection}). The parameter $\tau$ is commonly seen in the change-point literature, and is needed to threshold the maximum CUSUM statistics and controls false positive detection. To derive $\tau$, we need to study the order of the maximum CUSUM statistics under the no change-point scenario. As for the tuning parameter $\zeta$, intuitively, for small subsamples, the estimation error of LGS and the CUSUM statistics become difficult to control. The parameter $\zeta$ is designed to handle such scenario and regulates the minimum length required for a subsample $(s,e]$ to be considered for change-point detection. See similar tuning parameters in \cite{leonardi2016computationally} and \cite{kaul2018parameter}. Note that simple algebra gives that $\Delta \succeq \zeta.$ In practice, it suffices to set $\zeta=\log(n)$.

In general, VPWBS is highly robust to the choices of $M$ and $\zeta$, and the key tuning parameters affecting the performance of VPWBS are $\lambda$ and $\tau$. In Section \ref{sec:sim}, we propose a cross-validation procedure to select $\lambda$ and $\tau$ in a fully data-driven fashion.

\section{Simulations}\label{sec:sim}
In this section, we conduct extensive numerical experiments to examine the performance of VPWBS under various simulation settings and further compare it with  two other state-of-the-art methods in the literature, specifically, EBSA in \cite{leonardi2016computationally} and  SGL in \cite{zhang2015multiple}. Implementations of the numerical experiments can be found at the GitHub link \href{https://github.com/darenwang/VPBS}{here}. We discuss the implementation details such as settings for each algorithm and estimation accuracy metrics in Section \ref{subsec:alg_setting} and present the simulation results in Section \ref{subsec:simu_results}.

\subsection{Implementation details}\label{subsec:alg_setting}
Given estimated change-point estimators  $\{\widehat \eta_k\}_{k=1}^{\widehat K}$, we measure the estimation accuracy via the scaled Hausdorff distance, a popular metric used in the change-point literature. Specifically, denoting the true change-points as $\{\eta_k\}_{k=1}^K $, the scaled Hausdorff distance is defined as
$$\mathcal D(\{\widehat \eta_k\}_{k=1}^{\widehat K} , \{\eta_k\}_{k=1}^K) = {d(\{\widehat \eta_k\}_{k=1}^{\widehat K}, \{\eta_k\}_{k=1}^K)}/{n},$$
where $ d(\cdot, \cdot)$ denotes the Hausdorff distance between two compacts sets $A,B$ in $\mathbb{R}$, given by
$$d(A ,B) = \max\left\{   \max_{a  \in A } \min_{b\in B }   |a-b| , \, \max_{b\in B } \min _{a  \in A } |a-b|\right\}.$$
Note that  $\mathcal D(\{\widehat \eta_k\}_{k=1}^{\widehat K} , \{\eta_k\}_{k=1}^K) \le 1$ when both $K ,\widehat K \ge 1 $. Therefore, following the convention in the change-point literature, we set $\mathcal D (\varnothing, \{\eta_k\}_{k=1}^K) = 1$. 

\textbf{Implementation of VPWBS}: As discussed in Section \ref{sec:vpwbs}, there are four tuning parameters $(M,\lambda,\tau,\zeta)$ in VPWBS. Throughout the simulation section, we set $M=40$ and set $\zeta=5$, which roughly corresponds to $M=(\log(n))^2$ and $\zeta=\log(n)$ across all simulation settings in Section \ref{subsec:simu_results}. We remark that the performance of VPWBS is robust to the choices of $(M, \zeta)$, and the key tuning parameters are $(\lambda, \tau)$.

In the following, we provide a sample splitting based cross-validation procedure that selects $(\lambda, \tau)$ in a fully data-driven fashion. Specifically, given the original sample $\{x_t,y_t\}_{t=1}^n$, we set the training data to be the oddly-indexed observations $\{x_{2t-1}, y_{2t-1}\}_{t=1}^{n/2}$ and the test data to be the evenly-indexed observations $\{x_{2t}, y_{2t}\}_{t=1}^{n/2}$, where we assume, without loss of generality, $n$ is even. Note that the training data and test data share the same number and locations of change-points~(up to one time point difference).

Denote the candidate sets of $\lambda,\tau$ as $\Lambda, \mathcal T \subset \mathbb R^+$. For each pair of $(\lambda,\tau) \in \Lambda\times \mathcal T$, using the training data, we compute the estimated change-points via VPWBS and further estimate the (piecewise constant) regression coefficients $\{\widehat\beta_t\}_{t=1}^{n/2}$ \textit{conditional} on the estimated change-points as in Model \eqref{assume: model assumption}. We then compute the prediction error of $\{\widehat \beta_t\}_{t=1}^{n/2}$ using the test data via
$$e_t=y_{2t}-x_{2t}^\top \widehat{\beta}_t, \quad t=1,2,\cdots,n/2.$$
The tuning parameters $(\lambda, \tau)$ are then selected as the pair of $(\lambda,\tau) \in \Lambda\times \mathcal T $ that achieves the minimum squared prediction error $\sum_{t=1}^{n/2}e_t^2$ on the test data. Note that in the cross-validation procedure, the random intervals $\{(a_m,b_m]\}_{m=1}^M$ and the minimum length $\zeta$ are kept the same across all pairs of $(\lambda,\tau)$. For all simulation experiments in Section \ref{subsec:simu_results}, we set $\Lambda=\{0.5,1,1.5,2\}$ and $\mathcal{T}=\{1,4,7,10,\cdots,49\}$.

\textbf{Implementation of competing methods}: The EBSA algorithm is proposed in \cite{leonardi2016computationally}, which performs change-point detection for high-dimensional regression via a model selection point of view. Specifically, a dynamic programming algorithm is proposed to directly estimate the unknown change-points by minimizing an $l_0$-penalized goodness of fit function. In contrast, VPWBS utilizes a group Lasso based local screening~(LGS) algorithm for estimating an optimal projection direction and uses CUSUM statistics for change-point estimation, where the LGS is an $l_1$-penalized M-estimator. For computational efficiency, a binary segmentation based algorithm is further proposed in \cite{leonardi2016computationally} to find an approximate minimizer of the penalized function with strong theoretical guarantees. We choose the tuning parameters of EBSA using its default settings as specified in \cite{leonardi2016computationally}. We note that EBSA gives slightly worse performance when its tuning parameters are selected via the sample splitting based cross-validation.

The sparse group Lasso~(SGL) is first introduced by \cite{simon2013sparse} and is later used by \cite{zhang2015multiple} for change-point detection in the high-dimensional regression setting. See also \cite{Harchaoui2007,Harchaoui2010}, \cite{Bleakley2011} and references therein for earlier work along this line of research, where the classical fused Lasso is used for change-point detection in mean for low-dimensional time series.

Given $\{x_t, y_t\}_{t=1}^n$, SGL computes
\begin{align}\label{eq:SGL full}
	\{\widehat \beta_t\}_{t=1}^n =&\arg\min_{ (\beta_1,\ldots,\beta_n)}\sum_{t=1}^n (y_t -X_t\beta_t )_2^2+ \lambda \sum_{t=1}^{n-1} \| \beta_{t+1} -\beta_{t}\|_2 + \gamma\sum_{t=1}^{n-1} \| \beta_{t+1} -\beta_{t}\|_1,
\end{align}
which can be seen as a variant of the classical fused Lasso with an extra group sparsity penalty.   Note that SGL is a global method as it estimates $\{\widehat \beta_t\}_{t=1}^n$~(and thus multiple change-points) based on the entire sample. In comparison, the local group Lasso based screening~(LGS) algorithm in Stage 1 of the proposed VPWBS is a local method and is designed to directly target single change-points. 

Define the function $f: \{1,2,\cdots,n-1\} \to \mathbb R$ where $f(t): = \| \widehat \beta_t -\widehat \beta_{t-1}\|_2.$ Note that $\{\widehat \beta_t\}_{t=1}^n$ estimated by SGL in \eqref{eq:SGL full} may not directly lead to accurate change-point estimation as $\sum_{t=1}^{n-1}\mathbb{I}(f(t)>0)$ is generally a large number and leads to uncontrollable false positives, where $\mathbb{I}$ denotes the indicator function.   In practice, the SGL estimator $\{\widehat \beta_t\}_{t=1}^n$ typically exhibits the so-called staircase pattern, a pattern commonly seen in fused Lasso based estimation~\citep[e.g.][]{Rojas2014,Owrang2017}, where $\{\widehat \beta_t\}_{t=1}^n$ contains large-scale changes accompanied by many small-scale jumps. See \Cref{fig:sgl} of the supplementary material for an illustration of such phenomenon.  To avoid false positive estimation, given knowledge of the true number of change-points $K$, a common practice in the literature, see e.g.\ \cite{Harchaoui2010}, is to estimate the change-points as the locations where the function $f$ achieves its $K$ largest values.

However, in practice, $K$ is typically unknown. Thus, to further improve the applicability of SGL, in our experiments we consider a variant of the SGL algorithm combined with wild binary segmentation in \cite{fryzlewicz2014wild}, which we refer to as Wild Binary Segmentation via SGL (WBSSGL). Specifically, WBSSGL further post-processes the estimated $\{\widehat{\beta}_t\}_{t=1}^n$ by SGL on $M$ random intervals $\{(a_m,b_m]\}_{m=1}^M$. For a subsample $\{\widehat{\beta}_t\}_{t=a_m+1}^{b_m}$ and $a_m\leq s_m<\nu<e_m\leq b_m$, WBSSGL computes the ($p$-dimensional) subsample CUSUM statistics for $\{\widehat{\beta}_t\}_{t=s_m+1}^{e_m}$ defined as
\begin{align}\label{eq:sgl_bs}
\widehat {\mathcal B} ^{s_m,e_m}_\nu =   \sqrt {\frac{ (e_m-\nu )}{ ( e_m-s_m) (\nu-s_m )} } \sum_{t=s_m+1}^{\nu} \widehat \beta _{t}  -\sqrt {\frac{(\nu-s_m ) }{ ( e_m-s_m)  (e_m-\nu ) } } \sum_{t=\nu+1}^{e_m  } \widehat \beta_{t},
\end{align}
and further compares it with a suitable threshold. The detailed implementation of WBSSGL is given in \Cref{algorithm:WBSSGL}. For all simulation experiments in Section \ref{subsec:simu_results}, numerical results indicate that WBSSGL outperforms the original SGL algorithm by a wide margin. Thus, in the following we only present the results for WBSSGL.

There are five tuning parameters $(M, \lambda, \gamma, \tau, \zeta)$ of WBSSGL, which are selected in the same way as VPWBS. Specifically, we set the random intervals $\{(a_m,b_m]\}_{m=1}^M$ and the minimum length $\zeta$ of WBSSGL to be the same as VPWBS. The key tuning parameters $(\lambda,\gamma,\tau)$ of WBSSGL are selected using the same cross-validation procedure as the one implemented for VPWBS. For each combination of $(\lambda,\gamma)$, we solve the original SGL in \eqref{eq:SGL full} via the R package \texttt{SGL}.

\begin{algorithm}[ht]
	\begin{algorithmic}
		\INPUT  data $ \{x_t, y_t\}_{t=1}^{n}  $, random intervals $\{(a_m,b_m]\}_{m=1}^M$, \\ tuning parameters $\lambda>0 , \gamma>0, \tau>0, \zeta>0$.
		
		\vspace{0.2cm}
		\noindent \hspace{-0.4cm} \textbf{Initialize} the set of estimated change-points as $\textbf{S}=\varnothing$ and set $(s,e]=(0,n]$.
		\vspace{0.1cm}
		
		\noindent \hspace{-0.4cm} \textbf{Stage 1}: SGL
		\State Compute $\{\widehat \beta _t \}_{t=1}^n $ via \eqref{eq:SGL full} with tuning parameters $\lambda, \gamma$.		
		
		\vspace{0.15cm}
		\noindent \hspace{-0.4cm} \textbf{Stage 2}: WBS($(s,e], \{ (a_m,b _m]\}_{m=1}^M, \tau ,\zeta$ )
		\For{$m = 1, \ldots, M$}  
		\State $ ( s_m, e_m] \leftarrow  ( s,e]\cap (a_m,b_m]$	
		\If {$e_m- s_m  \geq  2\zeta $}
		\State $D_{m} \leftarrow \arg\max_{s_m +\zeta \leq t \leq e_m -\zeta}  \| \widehat {\mathcal B} ^{s_m,e_m}_t  \|_2$  \Comment{Recall \Cref{eq:sgl_bs}}
		\State 	$A_m \leftarrow \max_{s_m +\zeta \leq t \leq e_m -\zeta}\| \widehat {\mathcal B} ^{s_m,e_m}_{t} \|_2$
		\Else
		\State
		$A_m \leftarrow   -1$
		\EndIf 
		\EndFor 
		\State $m^* \leftarrow \arg\max_{m = 1, \ldots, M} A_{m}$
		\If{$A_{m^*} > \tau$}
		\State add $D_{m^*}$ to the set $\textbf{S}$
		\State  WBS $(   (s, D_{m^* }] ,  \{ (a_m,b _m]\}_{m=1}^M , \tau ,\zeta  )$
		\State  WBS $(  (D_{m^* } ,  e ] ,  \{ (a_m,b _m]\}_{m=1}^M , \tau, \zeta  )$
		\EndIf
		
		\vspace{0.2cm}
		\OUTPUT The set of estimated change-points $\textbf{S}$.
		\caption{Wild Binary Segmentation via SGL. WBSSGL$(\{(a_m,b _m]\}_{m=1}^M , \lambda, \gamma, \tau, \zeta)$} \label{algorithm:WBSSGL}
	\end{algorithmic}
\end{algorithm}

\subsection{Simulation results}\label{subsec:simu_results}
In this section, we conduct extensive numerical experiments to examine the performance of VPWBS, EBSA and WBSSGL in terms of estimation accuracy and computational cost. We design a wide range of simulation settings by varying change size $\kappa$, spacing between change-points $\Delta$, number of change-points $K$, sparsity level $\s$, sample size $n$ and dimension $p$. The variance of noise $\sigma_\varepsilon^2$ is set at 1 for all settings. For each simulation setting, we repeat the experiments 100 times. The detailed simulation setting is as follows. We further plot typical realizations of $\{y_t\}_{t=1}^n$ for each setting in Figure \ref{figure:y}, where it can be seen clearly that information contained in $\{y_t\}_{t=1}^n$ is not sufficient for change-point estimation.

\textbf{Setting  (i): two change-points with varying change size $\kappa$}. In this setting, we fix $n=300$, $p=100$, $K=2$ and set the covariance matrix $\Sigma$ of $x_t$ to be the Toeplitz matrix with $\Sigma_{i,j} = 0.6^{|i-j|}$ for $i,j=1,\cdots,p$. The two change-points occur at $\eta_1=n/3=100$ and $\eta_2=2n/3=200$. The regression coefficients $\{\beta_t^*\}_{t=1}^n$ take the form
$$\beta_t^*  = \begin{cases}
	{\kappa} \cdot \alpha/{\sqrt{40} } , & \text{ for }   1 \le t  \le  {n}/{3},\\
	-{\kappa} \cdot  \alpha/{\sqrt{40} } , &\text{ for }  {n}/{3} +1   \le  t  \le  {2n}/{3},\\
	{\kappa}  \cdot \alpha /{\sqrt{40} } , & \text{ for }  {2n}/{3} + 1 \le t  \le n,
\end{cases}
$$
where $\alpha  = (\underbrace{1, -1,1,-1\ldots, -1}_{10}, \underbrace{0,\ldots,0}_{p-10})$ and we vary $\kappa \in \sqrt {40}\cdot \{ 1, 1.2, 1.4, 1.6\}$. Simple calculation shows that the change size of $\beta_t^*$ at both change-points equals $ \kappa$.

\textbf{Setting (ii): three change-points with varying sample size $n$}. In this setting, we vary $ n\in \{480, 560, 640, 720, 800 \}$, fix $p=100$, $K=3$ and set the covariance matrix $\Sigma$ of $x_t$ to be the identity matrix $I_p$. The three change-points occur evenly at $\eta_i=in/4$ for $i=1,2,3$. The regression coefficients $\{\beta_t^*\}_{t=1}^n$ take the form
$$\beta_t^* = \begin{cases}
	{2}/{5} \alpha,  &   \text{ for }    1 \le  t  \le  {n}/{4},	\\
	-{2}/{5} \alpha,  &    \text{ for }   {n}/{4} + 1 \le  t  \le  {n}/{2},	\\
	{2}/{5} \alpha,  &    \text{ for }   {n}/{2} + 1 \le  t  \le  {3n}/{4},	\\
	-{2}/{5} \alpha, &      \text{ for }  {3n}/{4} + 1 \le  t  \le  n,
\end{cases} 
$$
where $\alpha=( 1, -1 ,  1 ,-1   , 0,\cdots, 0)$. The change size of $\beta_t^*$ at each change-point equals $8/5.$

\textbf{Setting (iii): two change-points with varying $p$ and varying support of $\beta_t$}. In this setting, we fix $n=320$, $K=2$, vary $p\in \{  90, 100, 110, 120\}$ and set the covariance matrix $\Sigma$ of $x_t$ to be $I_p$. The two change-points occur unevenly at $\eta_1=120$ and $\eta_2=220$. The regression coefficients $\{\beta_t^*\}_{t=1}^n$ take the form 
$$\beta_t ^*  = \begin{cases}
	{2}/{3} \cdot (   \underbrace{1,\ldots,1}_{8}  , 0, \ldots, 0), &\text{ for } 1 \le  t\le 120,
	\\
	{2}/{3} \cdot (   \underbrace{0,\ldots,0}_{8}, \underbrace{1,\ldots,1}_{8}  , 0, \ldots, 0), &\text{ for  } 121 \le  t \le 220,
	\\
	{2}/{3} \cdot (   \underbrace{0,\ldots,0}_{16}, \underbrace{1,\ldots,1}_{8}  , 0, \ldots, 0), &\text{ for } 221 \le t \le 320.
\end{cases} 
$$
Simple calculation shows the change size of $\beta_t^*$ at each change-point equals $8/3.$

\textbf{Setting (iv): two change-points with uneven spacing and varying support  size $s$}. In this setting, we fix $n=520$, $K=2$, $p=100$ and set the covariance matrix of $x_t$ to be $\Sigma= I_p $. The two change-points occur unevenly at $\eta_1=160$ and $\eta_2=360$. The regression coefficients $\{\beta_t^*\}_{t=1}^n$ take the form 
$$\beta_ t^*  = \begin{cases}
	 \sqrt { {2}/{s}} \cdot  (   \underbrace{1,\ldots,1}_{s/2}  , \underbrace{3,\ldots,3}_{s /2 }  ,  0, \ldots, 0)  ,   &\text{ for } 1 \le  t  \le 160,
	\\
\sqrt { {2}/{s}} \cdot (   \underbrace{2,\ldots,2}_{s /2}  , \underbrace{1,\ldots,1}_{s /2}  ,  0, \ldots, 0)  ,  &\text{ for } 161 \le t \le 360,
	\\
 \sqrt { {2}/{s}}\cdot  (   \underbrace{1,\ldots,1}_{s /2}  , \underbrace{3,\ldots,3}_{ s /2}  ,  0, \ldots, 0) ,  &\text{ for } 361 \le t \le 520,
\end{cases}
$$
and we vary $ {s} \in \{16,20, 24, 28 \}$. Simple calculation shows that the change sizes of $\beta_t^*$ at both change-points equal $ \kappa=\sqrt {5}$.

\begin{center}
	\begin{figure}[h]
		\includegraphics[scale=0.45]{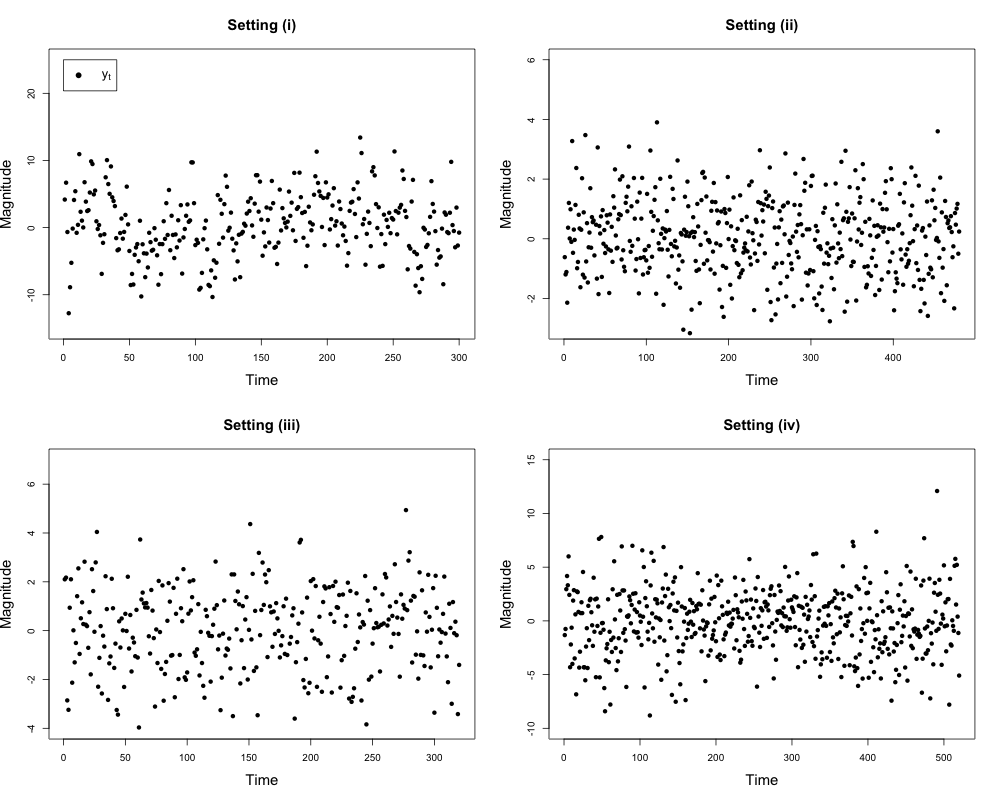} 
		\caption{Typical realizations of $\{y_t\}_{t=1}^n$ for Setting (i) with $\kappa=\sqrt{40}$, Setting (ii) with $n=480$, Setting (iii) with $p=90$ and Setting (iv) with $\kappa=0.35\sqrt{40}$.} \label{figure:y}
	\end{figure} 
\end{center}

\textbf{Estimation accuracy}: \Cref{table:numerical} reports the scaled Hausdorff distance (averaged over 100 repetitions) achieved by VPWBS, EBSA and WBSSGL across all simulation settings. For better visualization, Figure \ref{figure:bars} further provides the bar plots based on the results reported in Table \ref{table:numerical}. First, as expected, the performance of all three algorithms improve with larger sample size $n$~(setting (ii)) and with larger change size $\kappa$~(setting (i)), and worsen with higher dimension $p$~(setting (iii)) and with higher sparsity~(setting (iv)).

Overall, VPWBS offers robust and competitive performance for change-point estimation across all simulation settings, and consistently outperforms its competitors under the low SNR scenario where the sample size $n$ or the change size $\kappa$ is small. Compared to WBSSGL, which conducts change-point estimation directly on the estimated $\{\widehat{\beta}_t\}_{t=1}^n$ by a penalized M-estimator~(i.e.\ SGL), VPWBS in general gives more favorable performance, which could be seen as numerical evidence confirming the importance of the projection step in VPWBS. It is worth noting that when the SNR is large, EBSA becomes highly competitive. We conjecture that EBSA is also minimax optimal when the signal strength $\Delta \kappa^2$ is sufficiently large, though rigorous proof of such result seems challenging.

\begin{table}[ht]
	\begin{center}
		\renewcommand{\arraystretch}{1.5}
		\begin{tabular}{ccccccccc }   
			
			&   &  VPWBS        &   EBSA   & WBSSGL    
			\\
			\hline 
			\hline
			\multirow{4}{*}{Setting  (i)  } 
			& $\kappa=  \sqrt {40 } $ &   {\bf  0.011}    (0.015) & 0.090 (0.055) & 0.151 (0.094)
			\\  &$\kappa= 1.2\sqrt {40 } $ & {\bf 0.009}  (0.017) & 0.060 (0.052) &0.118 (0.069)
			\\ &$\kappa= 1.4\sqrt {40 } $& {\bf 0.010} (0.024)& 0.033 (0.044) &0.101  (0.058)
			\\ & $\kappa=1.6\sqrt {40 } $  & {\bf 0.009}  (0.016)& 0.025 (0.039) &0.098 (0.061)
			\\ 
			\hline
			\hline
			\multirow{5}{*}{Setting  (ii)  } 
			& $ n=  480 $ &   {\bf  0.062}    (0.080) & 0.128 (0.177) & 0.101 (0.052)
			\\  &$ n=   560  $ & {\bf 0.044}  (0.042) & 0.064 (0.112)  &0.098 (0.053)
			\\ &$n =  640 $& {\bf 0.034} (0.060)&  {\bf 0.034} (0.084) &0.094 (0.055)
			\\ & $n = 720  $  &0.026  (0.054)&  {\bf  0.015} (0.023)&0.093 (0.075)
			\\ & $ n = 800 $  &  0.022   (0.046)& {\bf 0.009} (0.010) &0.091 (0.072)
			\\
			\hline 
			\hline
			\multirow{5}{*}{Setting  (iii)  } 
			& $ p=  80 $ &   {\bf  0.025}    (0.047) & 0.052 (0.047) & 0.103 (0.065)
			\\
			&  
			$ p=  90 $ &   {\bf  0.039}    (0.068) & 0.068 (0.051) & 0.109 (0.068)
			\\  &$ p=   100  $ & {\bf 0.033}  (0.056) & 0.056 (0.050)  &0.120 (0.072)
			\\ &$p =  110 $& {\bf 0.041} (0.063)& 0.055 (0.046) &0.125  (0.072)
			\\ & $p = 120  $  & {\bf 0.049}  (0.078)& 0.061 (0.048)&0.140 (0.079)
			\\ 
			\hline 
			\hline
			\multirow{4}{*}{Setting  (iv)  } 
			& 	$s =  16$ &   {\bf  0.027}    (0.072) & 0.055 (0.023) & 0.116 (0.059)
			\\  &$s = 20 $ & {\bf 0.024}  (0.060) & 0.059 (0.024) &0.141 (0.033)
			\\ &$ s= 24 $& {\bf 0.043} (0.089)& 0.057 (0.020) &0.141  (0.032)
			\\ & $s =28 $  & {\bf 0.057}   (0.106)&    0.063  (0.022) &0.144 (0.023)
			\\
			\hline
			\hline
		\end{tabular}  
		\caption{Scaled Hausdorff distance for VPWBS,  EBSA (\cite{leonardi2016computationally}) and  WBSSGL (\cite{zhang2015multiple}). For each cell, the experiment is repeated 100 times. The numbers in the brackets indicate the sample standard errors of the scaled Hausdorff distance. Each highlighted number indicates the best performance in the corresponding setting. 
		} \label{table:numerical}
	\end{center} 
\end{table} 

	\begin{figure}[h]\begin{center}
		\includegraphics[scale=0.65]{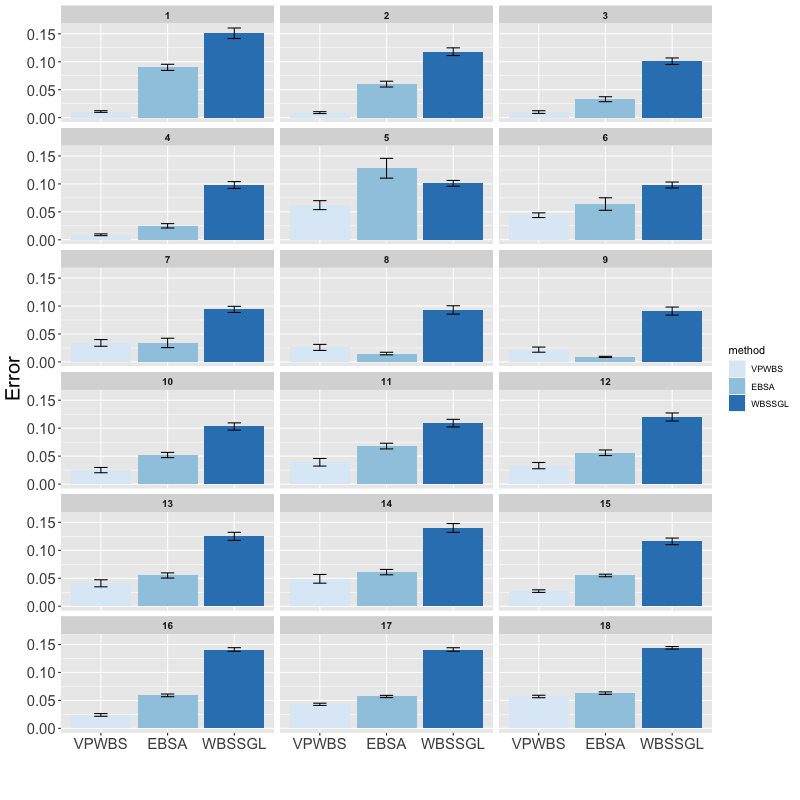} 
		\end{center}
		\caption{Bar plots for estimation results reported in \Cref{table:numerical}. Plots 1-4 correspond to Setting (i) with $\kappa\in \sqrt {40} \cdot\{1,1.2,1.4, 1.6\}$. Plots 5-9  correspond to Setting (ii) with $n\in  \{480, 560, 640, 720, 800 \}$. Plots 10-14  correspond to Setting (iii) with $p\in  \{80, 90, 100, 110, 120\}$. Plots 15-18  correspond to Setting (iv) with $s\in  \{   16,20, 24,28 \}$. } \label{figure:bars}
	\end{figure}

\textbf{Computational cost}: It is straightforward to derive that the computational cost of VPWBS is $O(M n\cdot \text{GroupLasso}(n,p))$, where $M$ is the number of random intervals used and recall that $\text{GroupLasso}(n,p)$ denotes the computational cost of the group Lasso for a $p$-dimensional regression with $n$ samples. Since we set $M=(\log (n))^2$, the computational complexity of VPWBS equals $O(n(\log(n))^2\cdot \text{GroupLasso}(n,p))$. On the other hand, referring to Table \ref{table for comparison}, the complexity of EBSA and WBSSGL are $O(n\log(n)\cdot \text{Lasso}(n,p))$ and $O(\text{Lasso}(n,np))$ respectively.

Note that the computational cost of solving Lasso and group Lasso for a $p$-dimensional linear regression with $n$ observations is both $O(np^2)$, see for example \cite{Efron2004} and \cite{Wright2009}. Thus, it is easy to see that in terms of computational efficiency, ESBA is the best, VPWBS comes second, and WBSSGL comes last. In practice, popular R packages typically implement gradient or coordinate descent to obtain an approximate solution of Lasso and group Lasso and the computation can be much faster than $O(np^2)$.

We conduct further numerical experiments to exam the computational performance of each algorithm in practice. Specifically, given $(n,p)$, we generate the regression coefficients $\{\beta_t^*\}_{t=1}^n$ and observations $\{x_t,y_t\}_{t=1}^n$ using Setting (i) with $\kappa= 1.6\sqrt {40}$. In the first set of experiments, we fix $n=450$ and vary $p\in \{80,100,120,140, 160, 180,200, 220 \}$; in the second set of experiments, we fix $p=100$ and vary $n\in \{240, 300, 360, 420, 480, 540,600, 660\} $. For each simulation setting, we repeat the experiments 100 times and report the average execution time of VPWBS, EBSA and WBSSGL in \Cref{figure:pop-projected}.

As can be seen in  \Cref{figure:pop-projected}, the computational costs of VPWBS and EBSA increase linearly with both the dimension $p$ and the sample size $n$. On the other hand, while the computational cost of WBSSGL grows linearly with $p$, it does not scale well with $n$. This is not surprising, as the SGL approach~\eqref{eq:SGL full} is essentially solving a Lasso with $n$ samples and $n p$ covariates.

\begin{center}
	\begin{figure}[H]
		\includegraphics[scale=0.57]{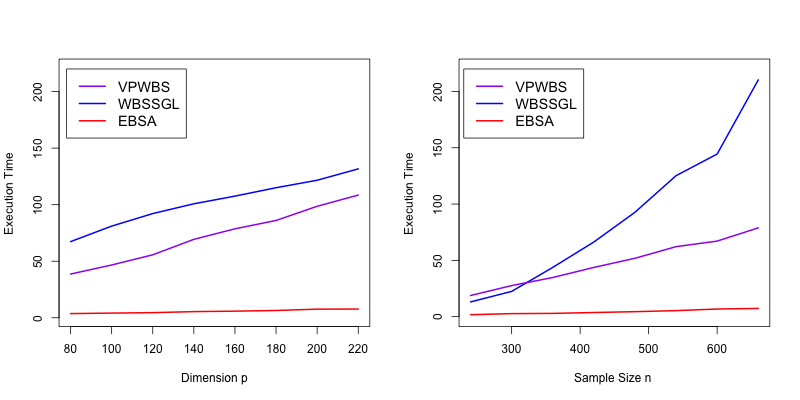} 
		\caption{Average execution time of VPWBS, EBSA and WBSSGL across different simple sizes $n$ and dimensions $p$.} \label{figure:pop-projected}
	\end{figure} 
\end{center}

\section{Discussion}\label{sec:discussion}
In this paper, we study the problem of multiple change-point estimation in high-dimensional linear regression model. We propose a novel projection-based algorithm, VPWBS, which performs change-point detection from a dimension reduction angle. Based on an estimated (optimal) projection direction, VPWBS transforms the original (difficult) problem of change-point detection in $p$-dimensional regression to a simpler problem of change-point detection in mean of a one-dimensional time series. VPWBS is shown to achieve the minimax optimal localization rate $O_p(1/n)$ up to a log factor, a significant improvement from the best rate $O_p(1/\sqrt{n})$ known in the existing literature. In addition, VPWBS is computationally efficient with a complexity of $O(n(\log(n))^2\cdot \text{GroupLasso}(n,p))$. Extensive numerical experiments are conducted to demonstrate the robust and favorable performance of VPWBS over two state-of-the-art algorithms in a wide range of simulation settings.

Besides the high-dimensional regression problem, we believe the projection based change-point estimation framework can be useful under other important contexts as well, such as change in covariance matrices and tensors. The key is to design an algorithm that utilizes the structure of the specific problem~(such as sparsity or low rank) and provides a provably accurate estimation of the (optimal) projection direction.

\subsection*{Acknowledgments}
We would like to thank the editor, Dr.\ Zaid Harchaoui, as well as the  three anonymous reviewers for their thoughtful assessment and constructive comments which helped us to improve the quality and the presentation of our paper.
  The work of RW was supported in part by AFOSR FA9550-18-1-0166, NSF DMS-1925101, NSF OAC-1934637, and DOE DE-AC02-06CH11357.
\bibliography{citations}

\appendix

\input{appendix}

\end{document}

%% file: appendix.tex
\section{Proofs Related to \Cref{corollary:property of projection}}
In this section, we provide all the technical details for the proof of \Cref{corollary:property of projection}

 \begin{proof}[Proof  of \Cref{corollary:property of projection}  ] Denote 
 $$\kappa'=\|  \beta^*_\eta -\beta^*_{\eta+1 }  \| _2.$$ 
 Since $\left(\widehat{\alpha}_1, \widehat{\alpha }_2, \widehat{ \nu } \right)$ be the output of  
	     LGS $( \{  y_ i , x_i  \}_{ i = 1}^{ n } ,  (s,e] ,\lambda   ) $, it holds that 
	\begin{align*} (  \widehat \alpha_1,  \widehat \alpha_2) =  \argmin_{  \alpha _1, \alpha _2 \in \mathbb{R}^{p}   }   \Bigg\{&\sum_{ i = s  + 1}^{ \widehat  \nu } ( y_{ i } - x_i ^{\top} \alpha_1 \bigr )^2  
			  +  \sum_{ i  =   \widehat \nu  + 1}^{e  }  ( y_ i  - x^\top_i  \alpha _2  ) ^2  \\
		   +&   \lambda \sum_{ j  = 1}^p \sqrt{( \widehat  \nu  - s )(\alpha _{1,j}) ^2 + (e -  \widehat  \nu )(\alpha _{2,j})^2}\Bigg\} \ , 
		   \end{align*} 
 and that  by \Cref{lemma:group lasso localization},
$$ |\eta -\widehat \nu  |\le C\frac{ \s\log(pn)  }{(\kappa^{' })^2 } . $$ 
 Without loss of generality, assume that $s<\eta < \widehat \nu \le e. $ 
 Then by \Cref{lemma:group lasso estimation},
 $$   \big  \|  \widehat \alpha _1 - \beta^*_{ (s, \widehat \nu  ]} \big\|_2 ^2 \le C\frac{\s \log(pn) }{\Delta}  \le  
 \frac{C\s\log(pn) }{ C_{snr}  \s \log(pn)  \kappa^{-2}}  \le \frac{C \kappa^2 }{C_{snr}} \le \frac{C  (\kappa') ^2 }{C_{snr}}   , $$
For sufficient large constant $C_{snr} $, it holds that
 $$  \big  \|  \widehat \alpha _1 - \beta^*_{ (s, \widehat \nu  ]} \big\|_2 \le \frac{c_x  }{128C_x  }  \kappa' \quad \text{and} \quad  \big  \|  \widehat \alpha _2 - \beta^*_{ (  \widehat \nu, e   ]}  \big\|_2 \le \frac{c_x  }{128C_x  } \kappa' .$$
Since 
$ \beta^*_{ (  \widehat \nu, e   ]} =\beta_{\eta+1}^*, $
it follows  that 
\begin{align} \label{eq:corollary projection deviation 1} \big  \|  \widehat \alpha _2 - \beta^*_{ \eta+1 }  \big\|_2 \le 
\frac{c_x  }{128C_x  }   \kappa' .
\end{align} 
Note that
\begin{align*}
 \big  \|  \widehat \alpha _1 - \beta^*_{ (s, \widehat \nu  ]} \big\|_2 = & \bigg  \|  \widehat \alpha _1 -  \frac{   (s-\eta) \beta^*_{ \eta}  +( \widehat \nu -\eta )\beta^*_{\eta+1 }  }{\widehat \nu -s  }\bigg\|_2
 \\
 \ge&  \|  \widehat \alpha _1 - \beta_\eta^*\|_2   - \bigg  \| \beta_\eta^* -  \frac{   (s-\eta) \beta^*_{ \eta}  +( \widehat \nu -\eta )\beta^*_{\eta+1 }  }{\widehat \nu -s }\bigg\|_2 
 \\
 \ge &   \|  \widehat \alpha _1 - \beta_\eta^*\|_2  - \frac{ \widehat \nu -\eta }{\widehat \nu -s }  \|\beta_{\eta}^* -\beta^*_{\eta+1} \|_2 
 \\
 \ge &  \|  \widehat \alpha _1 - \beta_\eta^*\|_2  - \frac{ C \s\log(pn)   / (\kappa^{' })^2   }{   { \frac{1}{20} }   C_{snr}   \s \log(pn)  /\kappa^2   }  (\kappa') 
 \\
 \ge &  \|  \widehat \alpha _1 - \beta_\eta^*\|_2   -\frac{c_x  }{128C_x  } \frac{ \kappa^2}{\kappa' }
 \\ 
 \ge &    \|  \widehat \alpha _1 - \beta_\eta^*\|_2   -\frac{c_x  }{128C_x  }  \kappa' ,
\end{align*}
where the third inequality follows from the fact that 
$\widehat \nu -s \ge   \frac{e-s  }{10}  \ge \frac{\Delta }{20 }  \ge  \frac{1}{20 }     C_{snr}  \kappa^{-2} \s \log(pn)    $ and  the fourth inequality holds if $C_{snr}$ is sufficiently large.
As a result 
\begin{align} \label{eq:corollary projection deviation 2}\|  \widehat \alpha _1 - \beta_\eta^*\|_2 \le \frac{c_x  }{64C_x  } \kappa' .
\end{align} 
 The desired result is an immediate consequence of \Cref{eq:corollary projection deviation 1} and \Cref{eq:corollary projection deviation 2}.
 \end{proof}

	     \begin{lemma} \label{lemma:group lasso localization}
	    Suppose that $[s+1, e] \subset [1,n] $  is any interval  such that $ e-s\ge \frac{\Delta}{2}$ and  that $[s+1, e]$  contains exactly one change point $\eta$ which satisfies 
	$$\min \{ \eta-s, e-\eta  \} \ge \frac{e-s }{10}.$$ 
	   	Suppose Assumptions \ref{assume: model assumption}-\ref{assume: LGS assumption} hold  and  $\lambda = C_\lambda \sqrt {\log(pn) } $ for a  sufficiently large constant $C_\lambda$.   Let   $\left(\widehat{\alpha}_1, \widehat{\alpha }_2, \widehat{ \nu } \right)$ be the output of  
	     LGS $( \{  y_ i , x_i  \}_{ i = 1}^{ n } ,  (s,e] ,\lambda  ) $.  Denote 
$$ \kappa ' = \|\beta^*_{\eta }   -\beta^*_{\eta  + 1} \| _2 .$$	      
	     Then with probability at least $1-(pn)^{-4}$, it holds that 
	     $$ |\eta -\widehat \nu  |\le C\frac{ \s\log(pn)  }{(\kappa^{' })^2 } . $$ 
	     \end{lemma}

\begin{proof}[Proof of \Cref{lemma:group lasso localization}]
Let $S_1 $ be the support of $\beta_i^*$ when $ i\in(s,\eta]$ and $S_2 $ be the support of $\beta_i^*$ when $ i\in (\eta,e]$. 
Denote 
$$ S=S_1\cup S_2 .$$
Note that $|S|\le 2\s$ and that $S$ is the common support for $\beta_i^* $ for $i\in(s,e]$. 
Without loss of generality, we assume that $s  < \eta < \widehat{\nu }  < e $.  
Denote 
$$
		\widehat{\beta}_i  = \begin{cases}
 				\widehat{\alpha}_1, &  i \in  (s   , \widehat{ \nu }  ] , \\
 				\widehat{\alpha }_2, & i  \in  (\widehat{\nu}    , e  ]. 
 			\end{cases}		
$$ 
If 
	$
		\widehat \nu- \eta  <C_1\frac{ \s\log(pn)  }{\kappa^2 } 
	$
	for some sufficiently large constant $C_1$, then the desired result holds. 
	\\
	\\
 So suppose that 
	\begin{align}\label{eq:group lasso spacing 1}
	\widehat \nu- \eta  \ge C_1\frac{ \s\log(pn)  }{\kappa^2 }   .
	\end{align} 
	Note that since $\kappa<C_\kappa<\infty$, \Cref{eq:group lasso spacing 1} implies that 
\begin{align}\label{eq:group lasso spacing 2}
	\widehat \nu- \eta  \ge C_1C_{\kappa}^{-2} \s\log(pn)  .
	\end{align} 
	By assumption,
	\begin{align}\label{eq:group lasso spacing 3} \min \{ \eta -s ,e- \widehat \nu  \}\ge    \frac{ e-s }{10}  \ge \frac{\Delta}{20 }	 \ge   \frac{1 }{ 20}     C_{snr} C_{\kappa}^{-2}   \s \log(p n) .
	 \end{align} 
\
\\
\\
{\bf Step 1.}	In this step, it is shown that 
 with probability at least $1 - C(n  p)^{-5}$,
	\[
		\sum_{ i  = s  + 1}^{e }\| \widehat{\beta}_ i  - \beta^*_ i \|_2^2 \leq C_3 \s \lambda^2 .
	\]	
From \Cref{eq:group lasso algorithm},  it holds that
	\begin{align}\label{eq-lem19-pf-1}
		& \sum_{ i  = s + 1}^{e }   ( y_ i  - x_ i ^{\top}\widehat{\beta}_i) ^2 + \lambda  \sum_{ j = 1}^p \sqrt{\sum_{ i  = s  + 1}^{e } \bigl(\widehat{\beta}_ {i,j}  \bigr) ^2} \leq \sum_{  i  = s + 1}^{e}   ( y_{ i } - x _ i^{\top}\beta^*_ i   ) ^2 + \lambda  \sum_{ j = 1}^p \sqrt{\sum_{  i  = s + 1}^{e} \bigl(\beta^*_ {i,j}\bigr) ^2}.
	\end{align}
	Let $\delta_ i        = \widehat{\beta}_ i - \beta^*_ i $.  It holds that
	\[
		\sum_{  i = s + 1}^{e - 1}\mathbbm{1}\left\{\delta_i     \neq \delta  _{ i +1}\right\} = 2.
	\]
	\Cref{eq-lem19-pf-1} implies that
	\begin{align}\label{eq-lem19-pf-2}
		\sum_{  i = s + 1}^{e} \|\delta_i    ^{\top} x_i\|_2^2 + \lambda  \sum_{ j = 1}^p \sqrt{\sum_{ i  = s + 1}^{e} \bigl(\widehat{\beta}_ {i,j} \bigr) ^2}\leq 2\sum_{  i = s + 1}^{e} (y_i - x_i^{\top}\beta^*_ i )\delta_i    ^{\top} x_i + \lambda \sum_{ j = 1}^p \sqrt{\sum_{  i = s + 1}^{e} \bigl(\beta^*_ {i,j}\bigr) ^2}.
	\end{align}
\
\\
Note that since $S$ is the common support for $\beta_i^* $ for $i\in(s,e]$, it holds that 
	\begin{align}
		& \sum_{ j = 1}^p \sqrt{\sum_{  i = s + 1}^{e} \bigl(\beta^*_ {i,j}\bigr) ^2} - \sum_{ j = 1}^p \sqrt{\sum_{  i = s + 1}^{e} \bigl(\widehat{\beta}_{i,j} \bigr) ^2} \nonumber
		\\
		=&  \sum_{ j  \in S} \sqrt{\sum_{  i = s + 1}^{e} \bigl(\beta^*_ {i,j}\bigr) ^2} - \sum_{ j  \in S} \sqrt{\sum_{  i = s + 1}^{e} \bigl(\widehat{\beta}_{i,j} \bigr) ^2}  - \sum_{j \in S^c} \sqrt{\sum_{  i = s + 1}^{e} \bigl(\widehat{\beta}_{i,j} \bigr) ^2} \nonumber 
		\\
		\leq & \sum_{ j  \in S} \sqrt{\sum_{  i = s + 1}^{e} \bigl(\delta_{i,j}  \bigr)^2}  - \sum_{j \in S^c} \sqrt{\sum_{  i = s + 1}^{e} \bigl(\delta_{i,j}  \bigr)^2}.\label{eq-lem19-pf-3}
	\end{align}
	Note that for any $j\in[1, \ldots, p]$,  
  from \Cref{eq:group lasso spacing 2}, \Cref{eq:group lasso spacing 3} and \Cref{lemma:l infty bound for estimators}, it holds that 
	\begin{align}
	\label{eq:l infty bound in group lasso}\sup_{s< i\le e  } \frac{\delta _{i,j}}{  \sqrt { \sum_{  i = s + 1}^{e} \left(\delta_{i,j}\right)^2 }  } \le C_4\frac{1}{\sqrt { \s \log(pn) }} .
	\end{align}
As a  result, with probability at least $1 -  (n   p)^{-5}$ 
	\begin{align}
		& \left|\sum_{  i = s + 1}^{e} (y_i - x_i^{\top}\beta^*_ i )\delta_i    ^{\top} x_i\right| = \left|\sum_{  i = s + 1}^{e} \varepsilon_{ i } \delta_i    ^{\top} x_i\right| = 
		\bigg|  \sum_{ j = 1}^p \left\{\left ( \frac{\sum_{  i = s + 1}^{e} \varepsilon_{ i } \delta_{i,j} x_{i,j}}{\sqrt{\sum_{  i = s + 1}^{e} \left(\delta_{i,j}\right)^2}}\right)  \sqrt{\sum_{  i = s + 1}^{e} \left(\delta_{i,j}\right)^2}\right\} \bigg|  \nonumber \\
		\leq & \sup_{ j= 1, \ldots, p} \left|\frac{\sum_{  i = s + 1}^{e} \varepsilon_{  i  } \delta_{i,j} x_  {i,j} }{\sqrt{\sum_{  i = s + 1}^{e} \left(\delta_{i,j}\right)^2}}\right| \sum_{ j = 1}^p \sqrt{\sum_{  i = s + 1}^{e} \left(\delta_{i,j}\right)^2} 
		\leq  C_5  \sqrt {\log(pn) } \sum_{ j = 1}^p \sqrt{\sum_{  i = s + 1}^{e} \left(\delta_{i,j}\right)^2} \nonumber 
		\\
		  \leq&  (\lambda/4)\sum_{ j = 1}^p \sqrt{\sum_{  i = s + 1}^{e} \left(\delta_{i,j}\right)^2},  \label{eq-lem19-pf-4}
	\end{align}
	where the second inequality follows from \Cref{lem-wang-lem-3} and \Cref{eq:l infty bound in group lasso} if $C_5$ is a  sufficiently large constant, and the last inequality follows from $\lambda = C_\lambda \sqrt {\log(pn) } $.
\ 
\\
Combining \eqref{eq-lem19-pf-1}, \eqref{eq-lem19-pf-2}, \eqref{eq-lem19-pf-3} and \eqref{eq-lem19-pf-4} yields  
	\begin{equation}\label{eq-lem19-pf-5}
		\sum_{  i = s + 1}^{e} ( \delta_i    ^{\top} x_i ) ^2	 + \frac{ \lambda }{2}\sum_{j \in S^c} \sqrt{\sum_{  i = s + 1}^{e} 
 		\bigl(\delta_{i,j}  \bigr)^2} \leq \frac{3\lambda }{2}\sum_{ j  \in S} \sqrt{\sum_{  i = s + 1}^{e} \bigl(\delta_{i,j}  \bigr)^2}.
	\end{equation}
	\
	\\
{\bf Step 2.} To apply  restricted eigenvalue conditions,   let
	\[
	I= (s,e], \quad 	I_1 = (s, \eta ], \quad I_2 = (\eta , \widehat {\nu}], \quad I_3 = (\widehat {\nu}, e].
	\]
	Denote 
	$$\mathcal E _{I }: = \Bigg\{ v^\top  \bigg( \frac{1}{|I|} \sum_{i\in I }  x_ix_i^\top \bigg)   v \ge  \frac{1}{16  }   v^\top  \Sigma v  -C_r  \frac{\log(p) }{ |I| }\|v\|_1^2  \Bigg \}  .$$
	If     $C_1$ in \Cref{eq:group lasso spacing 2} and   $C_{snr} $ in \Cref{eq:group lasso spacing 3} are sufficiently large constants,  it holds that 
\begin{align}
\label{eq:group lasso spacing change point} 
\min\{|I_1|, \ |I_2| , \ |I_3|  \} \ge C_6 \s \log(np) 
\end{align}
for some sufficiently large $C_6$.  For $h=1,2,3$, denote 
$\mathcal E_{I_h }$ the same way as $\mathcal E _{I }$.
 So  by \Cref{theorem:restricted eigenvalues},
	 $P(  \mathcal E_{I_h }) \ge 1- (np)^{-5} $ for $h=1,2,3$.
	Denote 
	$\delta_{I_ h } = \delta_  i  $ for any $ i\in  I_ h $.
	With probability at least $1 - 3(n   p)^{-5}$, on the event $\bigcap_{ h = 1, 2, 3}\mathcal{E}_{I_h }$,
	\begin{align} \nonumber 
		& \sum_{  i = s + 1}^{e}  ( \delta_i    ^{\top} x_i ) ^2 = 
		\sum_{   h  = 1,2,3 }  \sum_{ i  \in I_ h }   ( \delta _{I_ h}^{\top}x_i ) ^2 
		  \\ \nonumber 
		\geq   
		& 
		\sum_{ h  = 1, 2, 3}  \frac{  |  I_ h| } { 16  }  \delta_{I_h  }^\top  \Sigma   \delta_{I_h }   - C _r \log(p) \|\delta _{I_ h }\|_1 ^2 
		 \\
		 \nonumber  
	=
		& 
		\sum_{ h  = 1, 2, 3} \frac{  |  I_ h| } { 16  }  \delta_{I_h  }^\top  \Sigma   \delta_{I_h }   - C_r   \log(p) \big(  \|\delta _{I_ h } (S)\|_1 +\|\delta _{I_ h } (S^c)\|_1 \big) ^2  \\ \nonumber 
			\ge 
		& 
		\sum_{ h  = 1, 2, 3} \frac{ c_x  |  I_ h| } { 16  }   \| \delta_{I_h  } \|_2^2   - 2 C_r  |S|   \log(p)   \|\delta _{I_ h } \|_2^2   -2 C_r   \log(p)   \|\delta _{I_ h } (S^c)\|_1 ^2    \\ 
		\geq  
		&
		\sum_{ h  = 1, 2, 3}\frac{ c_x  |  I_ h| } { 20  }   \| \delta_{I_h  } \|_2^2      -2 C_r   \log(p)   \|\delta _{I_ h } (S^c)\|_1 ^2   ,
		\label{eq:restricted eigenvalue group lasso 1}
	\end{align}
	where the last inequality follows from \eqref{eq:group lasso spacing change point} and $|S| \le 2\s$ and 
	$\| \delta_{I_h}(S) \|_1 = \sum_{j\in S } |\delta_{I_h,j} |$.
	\\
	\\
	{\bf Step 3.}
Note that
	\begin{align*}
		 \sqrt { 	\sum_{  h= 1, 2, 3}   \|\delta _{I_ h } (S^c)\|_1 ^2 }  = 
			& \sqrt{\sum_{ h = 1,2,3} \left(\sum_{j \in S^c} |\delta_{I_{h  } ,j }  |\right)^2 }  = \sqrt{\sum_{ h  = 1,2,3} \left(\sqrt{\frac{|I_ h |}{ |I_h |  }} \sum_{j \in S^c}|\delta_ {  I_ h , j}  | \right)^2}	 \\
			\le &  \min\{ | I_1|,  |I_2| , |I_3| \}  ^{-1/2 }\sqrt{\sum_{i  =s + 1 }^e \left(  \sum_{j \in S^c}|\delta_{i,j} | \right)^2}	 
		\\
		\leq &  \min\{ | I_1|,  |I_2| , |I_3| \}  ^{-1/2 }   \sum_{j \in S^c}\sqrt{\sum_{  i = s + 1}^{e} (\delta_{i,j})^2} 
		\\ 
		\leq &  3\min\{ | I_1|,  |I_2| , |I_3| \}  ^{-1/2 } \sum_{j \in S}   \sqrt{\sum_{  i = s + 1}^{e} (\delta_{i,j})^2} \\
		\leq & 3 \min\{ | I_1|,  |I_2| , |I_3| \}  ^{-1/2 }   \sqrt{ |S| \sum_{j \in S} \sum_{  i = s + 1}^{e} (\delta_{i,j})^2} 
		\\ 
		\leq 
		& \sqrt{ \frac{c_x}{200  C_r \log(  p)}} \sqrt{\sum_{  i = s + 1}^{e} \|\delta_i    \|_2^2} \quad ,
	\end{align*}
where the  second inequality follows from the generalized Minkowski's inequality,  the third inequality follows from \Cref{eq-lem19-pf-5}, which implies that 
$$  \sum_{j \in S^c} \sqrt{\sum_{  i = s + 1}^{e} \bigl(\delta_{i,j}  \bigr)^2} \leq 3\sum_{ j  \in S} \sqrt{\sum_{  i = s + 1}^{e} \bigl(\delta_{i,j}  \bigr)^2}, $$
and the last inequality follows from \Cref{eq:group lasso spacing change point}  for sufficiently large $C_6$.  
So 
\begin{align}
 \sqrt { 	\sum_{i = 1, 2, 3}   \|\delta _{I_ h } (S^c)\|_1 ^2 }   \le \sqrt{ \frac{c_x}{200  C_r \log(  p)}} \sqrt{\sum_{  i = s + 1}^{e} \|\delta_i    \|_2^2} \  .  \label{eq:restricted eigenvalue group lasso 2}
\end{align}
Therefore, for any 
	\begin{align*}
		\eqref{eq:restricted eigenvalue group lasso 1} = & 		\sum_{ h  = 1, 2, 3} \frac{ c_x  |  I_ h| } { 20  }   \| \delta_{I_h  } \|_2^2      -2 C_r   \log(p)   \|\delta _{I_ h } (S^c)\|_1 ^2     \\
	 =
	 & \sum  _{ i=s+1} ^ e  \frac{ c_x    } { 20  } \| \delta_{ i  } \|_2^2  -2 C_r   \log(p)   \|\delta _{I_ h } (S^c)\|_1 ^2   
	 \\
	 \ge & \sum  _{ i=s+1} ^ e  \frac{ c_x    } { 20  } \| \delta_{ i  } \|_2^2  -2 C_r   \log(p)  \frac{c_x}{200 C_r  \log(  p) } \sum_{  i = s + 1}^{e} \|\delta_i    \|_2^2 
	 \\
	 \ge &  \sum  _{ i=s+1} ^ e  \frac{ c_x    } { 25  } \| \delta_{ i  } \|_2^2 \ ,
	\end{align*}
	where the second last inequality follows from  \Cref{eq:restricted eigenvalue group lasso 2}.
	 \\
	 \\
	 {\bf Step 4.} Putting the previous steps together, it holds that 
	 \begin{align*}
	  & \sum  _{ i=s+1} ^ e  \frac{ c_x    } { 25  } \| \delta_{ i  } \|_2^2 + \frac{ \lambda }{2}\sum_{j \in S^c} \sqrt{\sum_{  i = s + 1}^{e} \bigl(\delta_{i,j}  \bigr)^2} 
	  \\ 
	  \le  
	  &  \sum_{  i = s + 1}^{e} ( \delta_i    ^{\top} x_i ) ^2	 + \frac{ \lambda }{2}\sum_{j \in S^c} \sqrt{\sum_{  i = s + 1}^{e} \bigl(\delta_{i,j}  \bigr)^2}
	  \\
	   \leq
	   &  \frac{3\lambda }{2}\sum_{ j  \in S} \sqrt{\sum_{  i = s + 1}^{e} \bigl(\delta_{i,j}  \bigr)^2}
	   \\
	   \le & \frac{3\lambda }{2} \sqrt{ |S| \sum_{ j  \in S} \sum_{  i = s + 1}^{e} \bigl(\delta_{i,j}  \bigr)^2} 
	   \\
	   \le &   3 \lambda \sqrt { \s}    \sqrt{   \sum_{  i = s + 1}^{e} \|   \delta_i   \|_2^2 }  \  . 
	 \end{align*}
	Therefore,
\begin{align}\label{eq:step 3 group lasso localization}
	 \sum_{i=s+1} ^e  \| \widehat \beta _i  - \beta^*_i  \|_2^2 =   \sum  _{ i=s+1} ^ e    \| \delta_{ i  } \|_2^2 
	   \le &  C _7  \lambda ^2  \s  \ . 
	 \end{align} 
\\
\\
 {\bf Step 5.}  From \Cref{eq:step 3 group lasso localization} and  the definition of $\widehat \beta_i$, it holds that 
\begin{align}\label{eq:group lasso deviation bound 1}
		\sum_{i=s+1} ^e  \| \widehat \beta _i  - \beta^*_i  \|_2^2  = |I_1| \|\beta^*_{\eta   } - \widehat{\alpha }_1\|_2^2 + |I_2| \|\beta^*_{\eta  + 1}  - \widehat{\alpha }_1\|_2^2 + |I_3| \|\beta^*_{\eta  + 1}  - \widehat{\alpha }_2\|_2^2  \le C _7  \lambda ^2  \s  .
	\end{align}
	Let 
	$$\kappa' = \| \beta^*_{\eta   }    - \beta^*_{\eta + 1    } \|_2. $$
	By \Cref{eq:group lasso spacing 3} 
	$$ \min\{ |I_1| , \ |I_3| \}   \ge     \frac{  \Delta}{20}. $$
	Therefore 
	it holds that 
	$$ \|\beta^*_{\eta   }   - \widehat{\alpha }_1\|_2^2 \le \frac{C_7 \lambda ^2  \s  }{|I_1| } \le  \frac{C_7 \lambda ^2  \s  }{ \frac{1}{20} \Delta  } \le  
\frac{C_7 C_\lambda ^2  \log(p n ) \s }{ \frac{1}{20}C_ {snr}    \log(pn)  \s /\kappa^2   }	\le 
\frac{1}{16 } \kappa^2 ,  $$
where  the first  inequality follows from \Cref{eq:group lasso deviation bound 1},   the third inequality follows from Assumption \ref{assume: LGS assumption} {\bf b}  and the last inequality holds for sufficiently large $C_ {snr} $. 
 Similarly 
 $$ \|\beta^*_{\eta  + 1 }    - \widehat{\alpha }_2\|_2^2  \le 
\frac{1}{16 } \kappa^2  .  $$  
	 So 
	 \begin{align*}
	  \|  \beta^*_{\eta + 1    }  - \widehat{\alpha }_1   \|_2 \ge \|\beta^*_{\eta  + 1   }  - \beta^*_{\eta   }    \|_2   - \|\beta^*_{\eta   }  - \widehat{\alpha }_1\|_2 \ge \kappa'  -  \frac{1}{4 }\kappa    \ge \kappa'  /2. 
	 \end{align*} 
 \Cref{eq:group lasso deviation bound 1} further implies that 
\begin{align*}
 (\widehat  \nu - \eta ) (\kappa') ^2 /4  \le  |I_2|  \|\beta^*_{\eta   }  - \widehat{\alpha }_1\|_2 ^2   \le C_7 \lambda^2 \s ,
\end{align*}   
which implies that 
$$ \widehat  \nu - \eta   \le  \frac{ 4 C_7 \lambda^2 \s }{  ( \kappa')   ^2 } ,  $$
as desired. 
\end{proof}

	     \begin{lemma} \label{lemma:group lasso estimation}
	 Suppose that $[s+1, e] \subset [1,n] $  is any interval  such that $ e-s\ge \frac{\Delta}{2}$ and  that $[s+1, e]$  contains exactly one change point $\eta$ which satisfies 
	$$\min \{ \eta-s, e-\eta  \} \ge \frac{e-s }{10}.$$   Suppose 
			\begin{align*} (  \widehat \alpha_1,  \widehat \alpha_2) =  \argmin_{  \alpha _1, \alpha _2 \in \mathbb{R}^{p}   }   \Bigg\{&\sum_{ i = s  + 1}^{ \nu } ( y_{ i } - x_i ^{\top} \alpha_1 \bigr )^2  
			  +  \sum_{ i  =  \nu  + 1}^{e  }  ( y_ i  - x^\top_i  \alpha _2  ) ^2  \\
		   +&   \lambda \sum_{ j  = 1}^p \sqrt{( \nu  - s )(\alpha _{1,j}) ^2 + (e -  \nu )(\alpha _{2,j})^2}\Bigg\}.
		   \end{align*}
If  in addition, 	  Assumptions \ref{assume: model assumption}-\ref{assume: LGS assumption} hold and that $\lambda = C_\lambda \sqrt {\log(pn) } $ for sufficiently large $C_\lambda$,
	     then  with probability at least $1-(pn)^{-4} $, it holds that 
	     $$ \big  \|  \widehat \alpha _1 - \beta^*_{(s, \nu] } \big\|_2 ^2 \le C\frac{\s \log(pn) }{\Delta} \quad \text{and} \quad 
	      \big  \|  \widehat \alpha _2 - \beta^*_{(\nu, e] } \big\|_2 ^2 \le C\frac{\s \log(pn) }{\Delta} ,$$
	      where 
	      $$  \beta^*_{(a, b] } = \frac{1}{b-a} \sum_{i=a+1}^b \beta_i^*.$$ 
	     \end{lemma}

\begin{proof}[Proof of \Cref{lemma:group lasso estimation}]
Let $S_1 $ be the support of $\beta_{\eta }^*$   and $S_2 $ be the support of $\beta_{\eta+ 1 }^*$.
Denote 
$$ S=S_1\cup S_2 .$$
Note that $|S|\le 2\s$ and that $S$ is the common support for $\beta_i^* $ for $i\in(s,e]$. 
Without loss of generality,  assume that $s  < \eta <  \nu   < e $.  
	Note that  
	\begin{align}
	\label{eq:group lasso estimation spacing 1}   \min\{ \nu-s, e-\nu\}> \frac{\Delta}{20}    \ge  \frac{1}{20}  C_{snr} C_{\kappa}^{-2}   \s \log(p n) . \end{align} 
For brevity, denote 
$$I_1 = (s,\nu]   ,   \quad   I_2= (\nu,e],  \quad  \alpha_1^* = \beta^*_{(s, \nu] } \quad \text{and} \quad \alpha_2^* = \beta^*_{ (\nu, e] }   \  . $$
\
\\
{\bf Step 1.}
From \Cref{eq:group lasso algorithm},  it holds that
	\begin{align}\begin{split}
		\sum_{ i  \in I_1 } ( y_{ i } - x_i ^{\top}  \widehat \alpha_1 \bigr )^2  
			  +  \sum_{i  \in I_2  }  ( y_ i  - x^\top_i   \widehat \alpha _2  ) ^2 
		   +   \lambda \sum_{ j  = 1}^p \sqrt{ |I_1| ( \widehat \alpha _{1,j}) ^2 + |I_2| ( \widehat \alpha _{2,j})^2} 
		   \\
		   \le  \sum_{ i \in   I_1  } ( y_{ i } - x_i ^{\top}   \alpha_1 ^*  \bigr )^2  
			  +  \sum_{ i   \in I_2  }  ( y_ i  - x^\top_i    \alpha _2  ^*  ) ^2 
		   +   \lambda \sum_{ j  = 1}^p \sqrt{ |I_1| (   \alpha _{1,j}   ^*  ) ^2 +  |I_2| (   \alpha ^*  _{2,j} )^2}   . 
		   \end{split}\label{eq:gle minimization}
	\end{align}
	Let $\phi _ i        = \widehat{\alpha }_ i - \alpha^*_ i $.  
	\Cref{eq:gle minimization} implies that
	\begin{align}  \nonumber 
		&\sum_{ i  \in I_1  } (  x_i ^{\top}  \phi _1 \bigr )^2  +\sum_{ i  \in I_2  } (  x_i ^{\top}  \phi _2 \bigr )^2 		 
		   +   \lambda \sum_{ j  = 1}^p \sqrt{ |I_1| ( \widehat \alpha _{1,j}) ^2 + |I_2| ( \widehat \alpha _{2,j})^2} 
		   \\ \nonumber  
		   \le  &  2 \sum_{ i  \in I_1  } ( y_i -x_i^\top \alpha_1^* ) x_i ^{\top}  \phi _1  +
		  2  \sum_{ i  \in I_2  } ( y_i -x_i^\top \alpha_2^* ) x_i ^{\top}  \phi _2 
		   +   \lambda \sum_{ j  = 1}^p \sqrt{  |I_1| (   \alpha _{1  ,  j} ^* ) ^2 +  |I_2| (   \alpha ^*  _{2,j} )^2}  
		   \\ \label{eq:gle minimization 2}
		   = & 2  \sum_{ i  \in I_1  } ( \beta_i^* -\alpha^*_1  ) x_i x_i ^{\top}  \phi _1  +
		  2  \sum_{ i  \in I_2  } (\beta_i^*-\alpha^*_2  ) x_i x_i ^{\top}  \phi _2  
		\\ \label{eq:gle minimization 3}
		  + &    2  \sum_{ i  \in I_1  }  \varepsilon_i  x_i ^{\top}  \phi _1  +
		   2 \sum_{ i  \in I_2  }   \varepsilon_i   x_i ^{\top}  \phi _2    
		   \\
		   + &  \lambda \sum_{ j  = 1}^p \sqrt{  |I_1| (   \alpha _{1  ,  j} ^* ) ^2 +  |I_2| (   \alpha ^*  _{2,j} )^2}    .
		  \nonumber 
	\end{align} 
\
\\
Note that   $S$ is the common support for $\alpha_1^* $ and $\alpha_2^*$  for $i\in(s,e]$. So 
\begin{align*} \nonumber 
&  \sum_{ j  = 1}^p \sqrt{ |I_1| (   \alpha _{1  ,  j} ^* ) ^2 +  |I_2| (   \alpha ^*  _{2,j} )^2}   
 -
\sum_{ j  = 1}^p \sqrt{   |I_1|  ( \widehat \alpha _{1,j}) ^2 +  |I_2 | ( \widehat \alpha _{2,j})^2}  
\\ \nonumber 
=&\sum_{j \in S }\sqrt{ |I_1|   (   \alpha _{1  ,  j} ^* ) ^2 + |I_2 |  (   \alpha ^*  _{2,j} )^2}      -\sum_{  j \in S }  \sqrt{   |I_1|  ( \widehat \alpha _{1,j}) ^2 +  |I_2 | ( \widehat \alpha _{2,j})^2}   -\sum_{ j \in S^c}  \sqrt{   |I_1|  ( \widehat \alpha _{1,j}) ^2 +  |I_2 | ( \widehat \alpha _{2,j})^2}  
\\ \nonumber 
\le  
&  \sum_{j \in S }\sqrt{ |I_1|   (   \phi _{1,j} ) ^2 + |I_2 |  (   \phi _{2  ,  j} )^2}       -\sum_{ j \in S^c}  \sqrt{   |I_1|  ( \widehat \alpha _{1,j}) ^2 +  |I_2 | ( \widehat \alpha _{2,j})^2}  
\\
= &  \sum_{j \in S }\sqrt{ |I_1|   (   \phi _{1,j} ) ^2 + |I_2 |  (   \phi _{2  ,  j} )^2}       -\sum_{ j \in S^c}  \sqrt{   |I_1|  ( \phi_{1,j}) ^2 +  |I_2 | ( \phi  _{2,j})^2}   . 
\end{align*}
{\bf Step 2.}	Note that for any $h=1,2$ and  $j\in[1,p]$,  
  from \Cref{eq:group lasso estimation spacing 1}, it holds that 
	\begin{align}
	\label{eq:l infty bound in group lasso estimation} \frac{\phi  _{h,j}}{  \sqrt { |I_1|   (   \phi _{1,j} ) ^2 + |I_2 |  (   \phi _{2  ,  j} )^2   }  } \le \frac{1}{\sqrt {|I_h| }} \le  C_4\frac{1}{\sqrt { \s \log(pn) }} .
	\end{align}
	\
	\\
As a  result, with probability at least $1 - C(n   p)^{-5}$
	\begin{align*} \nonumber 
		 & | \eqref{eq:gle minimization 3}| 		
 =   
		  \Bigg| 
		 \sum_{ j = 1}^p \Bigg\{      \frac{   \sum_{  i \in I_1} \varepsilon_{ i } \phi_{1,j} x_{i,j}     
		 +  \sum_{  i \in I_2 } \varepsilon_{ i } \phi_{2  ,  j} x_{i,j}  }    {\sqrt { |I_1|   (   \phi _{1,j} ) ^2 + |I_2 |  (   \phi _{2  ,  j} )^2   }   }     \Bigg\} \sqrt { |I_1|   (   \phi _{1,j} ) ^2 + |I_2 |  (   \phi _{2  ,  j} )^2   }       \Bigg| \nonumber \\
		\leq & \sup_{ j= 1, \ldots, p}  \Bigg|   \frac{   \sum_{  i \in I_1} \varepsilon_{ i } \phi_{1,j} x_{i,j}     +  \sum_{  i \in I_2 } \varepsilon_{ i } \phi_{2  ,  j} x_{i,j}  }    {\sqrt { |I_1|   (   \phi _{1,j} ) ^2 + |I_2 |  (   \phi _{2  ,  j} )^2   }   }     \Bigg |     \sum_{   j  =1}^p \sqrt { |I_1|   (   \phi _{1,j} ) ^2 + |I_2 |  (   \phi _{2  ,  j} )^2   }        \nonumber \\ 
		\leq &    C_5  \sqrt {\log(pn) }   \sum_{ j  =1}^p \sqrt { |I_1|   (   \phi _{1,j} ) ^2 + |I_2 |  (   \phi _{2  ,  j} )^2   }     
		  \le (\lambda/8)\sum_{ j = 1}^p    \sqrt { |I_1|   (   \phi _{1,j} ) ^2 + |I_2 |  (   \phi _{2  ,  j} )^2   },    
	\end{align*}
	where the second inequality follows from \Cref{lem-wang-lem-3} and \Cref{eq:l infty bound in group lasso estimation}, and the last inequality follows from $\lambda = C_\lambda \sqrt {\log(pn) } $.
\ 
\\
In addition,  since  $ s< \eta <\nu <e,$ it holds that
$$\alpha _2^*  = \frac{1}{ (e-\nu ) } \sum_{i  =\nu+1 }^e  \beta_i ^* = \beta _{\eta+1 } ^* .$$ 
As a result 
\begin{align} \nonumber  
	 \eqref{eq:gle minimization 2} =  & 	  \sum_{ i  \in I_1  }  (\beta_i^*-\alpha_1^* )^\top   x_i x_i ^{\top}  \phi _1  +
		   \sum_{ i  \in I_2  }  (\beta_i^*-\alpha_ 2^* )^\top    x_i x_i ^{\top}  \phi _2 		  
		   \\ \nonumber   
		   = & \sum_{ i  \in I_1  }  \bigg(\beta_i^*-   \frac{1}{|I_1|} \sum_{ {i'} \in I_1 } \beta_{i'} ^* \bigg  )^\top   x_i x_i ^{\top}  \phi _1 
		   \\ \nonumber  
		   \le &   \sup_{1\le j \le p }\Bigg|  \sum_{ i  \in I_1  }  \bigg(\beta_i^*-   \frac{1}{|I_1|} \sum_{ {i'} \in I_1 } \beta_{i'} ^* \bigg  )^\top   x_i x_{i,j}   \Bigg| \| \phi _1  \|_1 
		   \\ \nonumber  
		   = &  \sup_{1\le j \le p }\Bigg|  \frac{1}{\sqrt{ |I_1|}  }\sum_{ i  \in I_1  }  \bigg(\beta_i^*-   \frac{1}{|I_1|} \sum_{ {i'} \in I_1 } \beta_{i'} ^* \bigg  )^\top   x_i x_{i,j}   \Bigg|   \sqrt {|I_1| }\| \phi _1  \|_1  
		   \\ \label{eq:gle step 2 term 2}
		   \le &\sup_{1\le j \le p }\Bigg|  \frac{1}{\sqrt{ |I_1|}  }\sum_{ i  \in I_1  }  \bigg(\beta_i^*-   \frac{1}{|I_1|} \sum_{ {i'} \in I_1 } \beta_{i'} ^* \bigg  )^\top   x_i x_{i,j}   \Bigg|  \sum_{ j = 1}^p    \sqrt { |I_1|   (   \phi _{1,j} ) ^2 + |I_2 |  (   \phi _{2  ,  j} )^2   }   \quad  . 
	\end{align}  
For any $j \in [1,p]$, denote $\Sigma [,j] $ to be the $j$-th column of $\Sigma$. Then
$$    \sum_{ i  \in I_1  }  \bigg(\beta_i^*-   \frac{1}{|I_1|} \sum_{ {i'} \in I_1 } \beta_{i'} ^* \bigg  )^\top \Sigma [,j] = 0.$$
In addition, it is straightforward to see that 
$$ \sup_{ i \in I_1 } \bigg\|\beta_i^*-   \frac{1}{|I_1|} \sum_{ {i'} \in I_1 } \beta_{i'} ^*   \bigg\| \le C_\kappa.$$ 
So 
 $  \bigg( \beta_i^*-   \frac{1}{|I_1|} \sum_{ {i'} \in I_1 } \beta_{i'} ^* \bigg  )^\top     x_i x_{i,j}  $ 
is  a sub-exponential random variable with parameter $  C_x^2 C_\kappa$  for any $i\in I_1$.
As a result 
\begin{align*}
& P \Bigg( \Bigg|  \frac{1}{\sqrt{ |I_1|}  }  \sum_{ i  \in I_1  }  \bigg(\beta_i^*-   \frac{1}{|I_1|} \sum_{ {i'} \in I_1 } \beta_{i'} ^* \bigg  )^\top   x_i x_{i,j}   \Bigg| \ge \delta  \text{ for all } 1\le j \le p \Bigg) 
\\
= &  P \Bigg(  \Bigg|  \sum_{ i  \in I_1  }  \bigg(\beta_i^*-   \frac{1}{|I_1|} \sum_{ {i'} \in I_1 } \beta_{i'} ^* \bigg  )^\top    \big\{ x_i x_{i,j}    - \Sigma[,j]\big\}    \Bigg|   \ge \delta \sqrt {|I_1| } \text{ for all } 1\le j \le p  \Bigg) 
\\
\le &  p\exp \bigg(-c \min\{  \frac{\delta^2 }{   4C_x^2C_\kappa  } , \frac{\delta \sqrt{|I_1| }}{ 2C_x \sqrt{ C_\kappa} }\}\bigg) 
\le p \exp \bigg(-c'  \min\{  \delta^2  ,  \delta \sqrt{|I_1|}   \}\bigg), 
\end{align*} 
where the second to the last inequality follows from standard sub-exponential tail bounds. 
Since $|I_1| \ge  \frac{\Delta}{20 }    \ge \frac{1}{20}  C_{snr} C_{\kappa}^{-2}   \s \log(p n) , $ letting $\delta =C_\delta \sqrt {\log(p)}$ for sufficiently large  
constant $C_\delta$, it holds that with probability at least $1-(pn)^{-5}$,
$$ \sup_{1\le i \le p }\Bigg|  \frac{1}{\sqrt{ |I_1|}  }  \sum_{ i  \in I_1  }  \bigg(\beta_i^*-   \frac{1}{|I_1|} \sum_{ {i'} \in I_1 } \beta_{i'} ^* \bigg  )^\top   x_i x_{i,j}   \Bigg| \le C_\delta \sqrt {\log(p)} \le \frac{1}{8 }\lambda. $$ 
Therefore 
\eqref{eq:gle step 2 term 2} gives 
$$  \eqref{eq:gle minimization 2} \le(\lambda/8)\sum_{ j = 1}^p    \sqrt { |I_1|   (   \phi _{1,j} ) ^2 + |I_2 |  (   \phi _{2  ,  j} )^2   }    .  $$
\
\\
{\bf Step 3.} Combing the previous two steps  gives 
\begin{align*}
 		&\sum_{ i  \in I_1  } (  x_i ^{\top}  \phi _1 \bigr )^2  +\sum_{ i  \in I_2  } (  x_i ^{\top}  \phi _2 \bigr )^2  + 
 		\lambda \sum_{ j \in S^c}  \sqrt{   |I_1|  (  \phi  _{1,j}) ^2 +  |I_2 | (  \phi  _{2,j})^2}  
 		\\
 		\le &  \lambda \sum_{j \in S }\sqrt{ |I_1|   (   \phi _{1,j} ) ^2 + |I_2 |  (   \phi _{2  ,  j} )^2}   +  \frac{\lambda}{4}   \sum_{ j = 1}^p     \sqrt { |I_1|   (   \phi _{1,j} ) ^2 + |I_2 |  (   \phi _{2  ,  j} )^2   }  .
\end{align*}
This gives
\begin{align} \label{eq:gle step 4 1}
 		\sum_{ i  \in I_1  } (  x_i ^{\top}  \phi _1 \bigr )^2  +\sum_{ i  \in I_2  } (  x_i ^{\top}  \phi _2 \bigr )^2 	  + 
 		 \frac{\lambda}{2} \sum_{ j \in S^c}  \sqrt{   |I_1|  ( \phi  _{1,j}) ^2 +  |I_2 | ( \phi   _{2,j})^2}  
 		\le   \frac{3\lambda }{2}  \sum_{j \in S }\sqrt{ |I_1|   (   \phi _{1,j} ) ^2 + |I_2 |  (   \phi _{2  ,  j} )^2}   .
\end{align}
Using exactly the same argument as in {\bf Step 4} and { \bf Step 5} in the proof of \Cref{lemma:group lasso localization},
it can be shown that 
\begin{align}\label{eq:gle step 4 2}
 		\sum_{ i  \in I_1  } (  x_i ^{\top}  \phi _1 \bigr )^2  +\sum_{ i  \in I_2  } (  x_i ^{\top}  \phi _2 \bigr )^2 	  \ge 
 		\frac{c_x}{25}  \big( |I_1| \|\phi_1\|_2^2 +  |I_2|  \|  \phi _2 \|_2^2  \big) .
\end{align} 
Therefore 
\begin{align} \nonumber 
 		    |I_1| \|\phi_1\|_2^2 +  |I_2|  \|  \phi _2 \|_2^2   \le &
 		    C_6  \lambda   \sum_{j \in S }\sqrt{ |I_1|   (   \phi _{1,j} ) ^2 + |I_2 |  (   \phi _{2  ,  j} )^2}   
 		   \\\nonumber 
 		   \le & C_6\lambda \sqrt{ |S|  \sum_{j \in S } \big\{  |I_1|   (   \phi _{1,j} ) ^2 + |I_2 |  (   \phi _{2  ,  j} )^2 \big\} }    
 		   \\\nonumber 
 		  \le   & C_6\lambda\sqrt{ |S|    |I_1| \|\phi_1\|_2^2 +  |I_2|  \|  \phi _2 \|_2^2  }      
 		   \\    \nonumber  	 
 		   \le & 2C_6\lambda  \sqrt {\s}    \sqrt {  |I_1| \|\phi_1\|_2^2 +  |I_2|  \|  \phi _2 \|_2^2  }  \quad , 
\end{align} 
where the first inequality follows from \eqref{eq:gle step 4 1} and \eqref{eq:gle step 4 2}. This directly gives 
$$  |I_1| \|\phi_1\|_2^2 +  |I_2|  \|  \phi _2 \|_2^2   \le   8 C^2_6 \lambda ^2 \s .$$
The desired result follows from the assumption that  
 $$  \min\{ |I_1| ,   |I_2| \}>  \frac{1 }{20}   \Delta  .$$

 \end{proof}

\subsection{Additional Technical Lemmas}	       
	       
\begin{lemma} \label{lemma:covering}
Let $\mathcal{R }$ be any linear subspace in $\mathbb{R}^n$ and $\mathcal{N}_{1/4}$	be a $1/4$-net of $\mathcal{R } \cap B(0, 1)$, where $B(0, 1)$ is the unit ball in $\mathbb{R}^n$.  For any $u \in \mathbb{R}^n$, it holds that
	\[
		\sup_{v \in \mathcal{R } \cap B(0, 1)} \langle v, u \rangle \leq 2 \sup_{v \in \mathcal{N}_{1/4}} \langle v, u \rangle,
	\]
	where $\langle \cdot, \cdot \rangle$ denotes the inner product in $\mathbb{R}^n$.
\end{lemma}

\begin{proof}
Due to the definition of $\mathcal{N}_{1/4}$, it holds that for any $v \in \mathcal{R } \cap B(0, 1)$, there exists a $v_k \in \mathcal{N}_{1/4}$, such that $\|v - v_k\|_2 < 1/4$.  Therefore,
	\begin{align*}
		\langle v, u \rangle = \langle v - v_k + v_k, u \rangle = \langle x_k, u \rangle + \langle v_k, u \rangle \leq \frac{1}{4} \langle v, u \rangle + \frac{1}{4} \langle v^{\perp}, u \rangle + \langle v_k, u \rangle,
	\end{align*}
	where the inequality follows from  $x_k = v - v_k = \langle x_k, v \rangle v + \langle x_k, v^{\perp} \rangle v^{\perp}$.  Then we have
	\[
		\frac{3}{4}\langle v, u \rangle \leq \frac{1}{4} \langle v^{\perp}, u \rangle + \langle v_k, u \rangle.
	\]
	It follows from the same argument that 
	\[
		\frac{3}{4}\langle v^{\perp}, u \rangle \leq \frac{1}{4} \langle v, u \rangle + \langle v_l, u \rangle,
	\]
	where $v_l \in \mathcal{N}_{1/4}$ satisfies $\|v^{\perp} - v_l\|_2 < 1/4$.  Combining the previous two equation displays yields
	\[
		\langle v, u \rangle \leq 2 \sup_{v \in \mathcal{N}_{1/4}} \langle v, u \rangle.
	\]
 \end{proof}
 \
 \\
For any vector $v \in \mathbb R^m$,  denote  $ \mathcal D(v)$ to be the number of change point of $v $. That is,
$$\mathcal D(v) = \sum_{i=1}^m  \mathbbm 1\{  v_i \not = v_{i+1}\}. $$

\begin{lemma}\label{lem-wang-lem-3}
	For data generated according to   Model \ref{assume: model assumption}, for any interval $I = (s, e] \subset \{1, \ldots, n\}$, it holds that for any $\delta > 0$ and any  $ j  \in \{1, \ldots, p\}$, 
	\[
		\mathbb{P}\left\{\sup_{\substack{v \in \mathbb{R}^{(e-s)}   \\ \|v\|_2 = 1 ,   \mathcal D (v) =2  }}  \left|\sum_{ i  = s+1}^e v_ i \varepsilon_i  x_ i[j]\right| > \delta \right\} \leq Cn   \exp \left\{-c \min\left\{ \delta^2 , \, \frac{\delta}{  \|v\|_{\infty}}\right\}\right\}.
	\]
\end{lemma}

\begin{proof} This is a standard covering lemma. We provide a proof for completeness. 
For any $v \in \mathbb{R}^{(e-s)}$ satisfying $\mathcal D(v) =  2 $ and $\|v\|_2=1$,  let $\eta <\eta'  $ be the change points of $v$. Then there are $(e-s)(e-s-1)/2$ possible choice of $\eta, \eta' $. For any $\eta, \eta' $, denote 
$$  \mathcal R (\eta, \eta'  ) =\{ w \in  \mathbb{R}^{(e-s)} , w_1=\ldots=w_{\eta}\not =w _{\eta+1} =\ldots = w_{\eta'} \not = w_{\eta'+1}= \ldots =w_{(e-s)}\}$$
Then $\mathcal R (\eta, \eta'  ) $ is a 3-dimensional subspace. Denote $ \mathcal N_{1/4} (\eta, \eta' )$ to be  the covering number a  the unit ball in 
$\mathcal R (\eta, \eta'  ) $. Then   $ \mathcal N_{1/4} (\eta, \eta' ) \le 9^2. $
  Therefore we have,
	\begin{align*}
		& \mathbb{P}\left\{\sup_{\substack{v \in \mathbb{R}^{(e-s)}   \\ \|v\|_2 = 1 ,   \mathcal D (v) = 2  }}  \left|\sum_{   i  = s+1}^e v_ i \varepsilon_i   x _ i [j]\right| > \delta \right\} \\
		\leq &  \frac{(e-s)(e-s-1) }{2} 9^{2} \sup_{ \eta\in(s,e], \  v \in \mathcal{N}_{1/4} (\eta,\eta')  }	\mathbb{P}\left\{\left|\sum_{ i = s+1}^e v_ i \varepsilon_ i  x_ i[j] \right| > \delta/2 \right\} \\
		\leq & Cn^2 \exp \left\{-c' \min\left\{\frac{\delta^2}{4C_x^2}, \, \frac{\delta}{2C_x \|v\|_{\infty}}\right\}\right\} \\
		\leq & Cn^2 \exp \left\{-c  \min\left\{ \delta^2 , \, \frac{\delta}{  \|v\|_{\infty}}\right\}\right\} ,
	\end{align*}
where the first inequality follows from \Cref{lemma:covering} and union bounds,  the second last inequality holds because for any fixed $v$, $v_i\varepsilon_i x_{i,j}$ is a sub-Exponential random variable with parameter bounded by $C_x.$
\end{proof}
	       
	       \begin{lemma}\label{lemma:l infty bound for estimators}
	       Suppose 
	       $$ v= ( \underbrace  {a, \ldots, a}_{K_1 },\underbrace  {b, \ldots, b}_{K_2 }, \underbrace  {c, \ldots, c}_{K_3 }) $$
	       and that $v\not= 0$.
	       Then 
	       $$ \bigg\| \frac{v }{\|v\|_2 }\bigg\|_\infty \le \frac{1}{ \sqrt { \min \{ K_1, K_2, K_3  \} }  } .$$
	       \end{lemma}
	       \begin{proof}
	       It suffices to show $a/\|v\|_2  \le \frac{1}{ \sqrt { \min \{ K_1, K_2, K_3  \} }  } .  $
	       If $a=0$, then this trivially holds. Otherwise
	       $$ \frac{a  }{\|v\|_2 }  =\frac{a} { \sqrt {a^2  K_1   +b ^2 K_2   +c ^2 K_3 }  } \le \frac{1}{\sqrt {K_1}}. $$
	       \end{proof}
	   \begin{theorem}
\label{theorem:restricted eigenvalues}
Suppose  $\{x_{i } \}_{1\le i \le n } \overset{i.i.d.} {\sim}  N_p (0, \Sigma )  $.  Let 
$ \widehat \Sigma =  \frac{1}{n} \sum_{i=1}^n x_ix_i^\top .$
Then there exists constants $c$ and $C$ such that for all $v\in \mathbb R^p $,
\begin{align*}
 v^\top  \widehat \Sigma v \ge \frac{1}{16}   v^\top  \Sigma v  -C_r  \frac{\log(p) }{n}\|v\|_1^2 
\end{align*}
with probability at least $1-\exp(-cn). $
\end{theorem}  
  \begin{proof}
  This is the well known restricted eigenvalue condition.  The proof cam be found in \cite{raskutti2010restricted}.
  \end{proof}

\section{Proofs Related to \Cref{theorem:vpcusum} }	       
 
For any univariate time series $\{ z_i\}_{i=1}^n $ and 	any $ 1\le s< t < e<n $, denote the CUSUM statistics as 
\begin{align*}
	\widetilde Z^{s ,e }_{t}  =  \sqrt \frac{ e-t  }{(e-s) (t-s)  }  \sum_{i=s+1}^t z_i  -  \sqrt \frac{   t-s  }{(e-s)(e-t )   }  \sum_{i=t+1}^e  z_i .
\end{align*}

\begin{proof} [Proof of \Cref{theorem:vpcusum}] Throughout the proof,  assume that event  
$\mathcal A ( \{y_i^{(2) } ,  x_i^{(2) } \}_{i=1}^n  ,   \{ u_m \}_{m=1}^M ,\xi = C_1 \sqrt { (\N + 1)\log(n) } )$   in \Cref{eq:event A}, event  
$ \mathcal B ( \{y_i^{(2) } ,  x_i^{(2) } \}_{i=1}^n  ,   \{ u_m \}_{m=1}^M ,  \xi = C_1 \sqrt { (\N + 1)\log(n) }  )$   in \Cref{eq:event B}, event $\mathcal M$ in \Cref{eq:event M} and the good event in \Cref{corollary:property of projection} (with data  $\{y_i^{(1) } ,  x_i^{(1) } \}_{i=1}^n$) hold. 
Denote 
$$\delta_k  = \frac{C ( \N+1)  \log(n)}{\kappa_k^2 }  \quad \text{and} 
\quad \delta_{\max  } = \frac{C ( \N+1)  \log(n)}{\kappa ^2   } .    $$
Since $\delta _k$ is  the desired   localization rate, by induction, it suffices to consider any generic $(s, e] \subset (0, n]$ that satisfies the following three conditions:
	\begin{align*}
		&\eta_{r-1} \le s\le \eta_r \le \ldots\le \eta_{r+q} \le e \le \eta_{r+q+1}, \quad q\ge -1; \\
		& \text{ either }   \eta_r-s\le \delta_r  \quad \text{or} \quad   s-\eta_{r-1} \le  \delta_{r-1};
		\\
		& \text{ either }  \eta_{r+q+1}-e \le \delta_{r+q +1 } \quad  \text{or}\quad   e-\eta_{r+q} \le \delta _{r+q}  .
	\end{align*}
	Here $q = -1$ indicates that there is no change point contained in $(s, e]$.

Observe  that under Assumption \ref{assume: LGS assumption}, for sufficiently large constant $ C_{snr}$, it holds that $\delta_{\max  }  < \Delta /4$.  Therefore, it has to be the case that for any true  change point $\eta_r\in (0, n]$, either $|\eta_r  -s |\le \delta_{r}  $ or $|\eta_r -s| \ge \Delta - \delta_{\max  }     \geq  \Delta /4 $. This means that $ \min\{ |\eta_r-e|, |\eta_r-s| \}\le \delta_{r}  $ indicates that $\eta_r$ is a detected change point in the previous induction step, even if $\eta_r\in (s, e]$.  We refer to $\eta_r\in (s,e]$ as an undetected change point if $ \min\{ \eta_r -s, \eta_r-e\} \ge  \Delta/4 $.

To complete the induction step, it suffices to show that VPWBS$( (s, e] , \{ (a_m,b _m)\}_{m=1}^M  ,\lambda, \tau , \zeta  )$   
\\
{\bf  (i)} will not detect any new change
point in $(s,e ]$ if
all the change points in that interval have been previous detected, and
\\
{\bf	(ii) }will find a point $D_{m*}$ in $(s,e]$ such that $|\eta_r-D_{m*}|\le \delta_r$ if there exists at least one undetected change point in $(s, e)$.
\\
\\
Let 
     $ \{\widehat{\alpha}_1 ^m , \widehat{\alpha }_2^m , \widehat{\nu }^m  \}  $  
    be the output of  $ \text{ LGS } ( \{  y_ i^{(1) } , x_i^{(1) }   \}_{ i = 1}^{ n } ,  (a _m,b _m) ,\lambda   ) $, 
     and  $$ u _  m  =   \frac{ \widehat{\alpha }_2^m -\widehat{\alpha}_1 ^m   }{\| \widehat{\alpha }_2^m -\widehat{\alpha}_1 ^m \| _2 }  \in \mathbb R^p \quad $$  for all  $1\le m \le M$ .  
	Since the intervals $\{(a_m,b_m)\}_{m=1}^M$ are sampled 
	independently from the  data, the rest
	of the argument is made on the event $\mathcal M$, which is defined in \Cref{eq:event M}   and  this event  has no effects on  the distribution of the data.
\\
\\
 {\bf Step 1.} 
Denote 
$$ f_i (u_m) =  E \{ z_i(u_m ) \}      \quad \text{and}  \quad  \widetilde f_t^{s_m, e_m} (u_m)   = E \{ \widetilde Z_t ^{a_m, b_m} (u_m)   \}   .$$ 
Note that 
$$  f_i (u_m) = u_m \Sigma \beta_i . $$ 
On the event $\mathcal M$, for any $\eta_k\in (0, n]$, without loss of
	generality, there exists \begin{align}
	\label{eq:spacing of vpcusum}a_k\in
	[\eta_k-3\Delta/4,\eta_k-\Delta/2] \quad 
	\text{and} 
	\quad b _k \in
	[\eta_k+\Delta/2,\eta_k+3\Delta/4].
	\end{align}   
 In this step, it is shown that   for each $k \in [1,\ldots,K]$, it holds that
	\begin{align} \label{eq:properties of vectors}
		\max_{  a_k  +\zeta  \le t \le b_k  -\zeta   }|    \widetilde f_t^{a_k , b _k }(u_ k )    | \ge  \frac{7c_x}{32} \sqrt {\Delta}   \kappa_k   ,
	\end{align}
	where
 $$u_k =      \frac{ \widehat{\alpha }_2^ k  -\widehat{\alpha}_1 ^k   } {\| \widehat{\alpha }_2^k -\widehat{\alpha}_1 ^k  \| _2}  , $$	 and 
	 $$ \{\widehat{\alpha}_1 ^ k  , \widehat{\alpha }_2^ k  , \widehat{\nu }^ k   \}  =  \text{ LGS } ( \{  y_ i^{(1) }  , x_i^{(1) }  \}_{ i = 1}^{ n } ,  (a _ k ,b _k ) ,\lambda   ). $$ 
	 
	By \Cref{eq:spacing of vpcusum}, $[a_k, b_k]$
	contains exactly one change point $\eta_k$.  Since $f_i(u_k ) $ is a    one dimensional population  time series, 
	it holds that 
\begin{align}\label{eq:one change point 1d cusum}
		\widetilde f_t^{a_k , b _k } (u_ k )    =
			\begin{cases}
				   \sqrt {\frac {t-a _k}{(b _k-a _k)(b _k-t)} }( b _k-\eta_k)    u_k ^\top\Sigma   ( \beta _{\eta_k} ^*  - \beta _{\eta_{k-1}}^* ), & a_m  <   t\le \eta_k, \\
				    \sqrt {\frac { b _k-t}{(b _k- a _k)(t-a  _k)} }(\eta_k-a _k)   u_k ^\top   \Sigma   ( \beta _{\eta_k}^*  - \beta _{\eta_{k-1}}^* ) , & \eta_k <  t\le b_m .
			\end{cases}
	\end{align}
	\\
	\\
	Let    $$ \{\widehat{\alpha}_1 ^ k  , \widehat{\alpha }_2^ k  , \widehat{\nu }^k   \}  \leftarrow  \text{ LGS } ( \{ y_i^{(1)} , x_i^{(1)} \}_{ i = 1}^{ n } ,  (a _ k ,b _k ) ,\lambda   )  . $$ 
	From \Cref{corollary:property of projection}, with probability at least $ 1 -n^{-5}$, it holds that 
	$$\big  \|   ( \widehat \alpha _1^ k   - \widehat \alpha _2 ^ k    )  - (\beta _{\eta_k} ^*  - \beta _{\eta_{k-1}}^* )  \big\|_2 \le
	 \frac{c_x  }{32 C_x }  \| \beta _{\eta_k} ^*  - \beta _{\eta_{k-1}}^*  \| _2 . $$  
	Since $ c_x <C_x$ by definition, it holds that 
	$$  \frac{32}{33  } \le \frac{ \|  \beta _{\eta_k} ^*  - \beta _{\eta_{k-1}}^* \| _2 }{ \|\widehat \alpha _1 - \widehat \alpha _2  \|  _2 }  \le \frac{32}{31} .$$
\
\\
As a result, 
\begin{align}\nonumber 
& u_k ^\top  \Sigma   ( \beta _{\eta_k} ^*  - \beta _{\eta_{k-1}}^* )\\
\nonumber 
 =
& \frac { ( \widehat \alpha _1 - \widehat \alpha _2 )^\top  } { \| \widehat \alpha _1 - \widehat \alpha _2 \| _2 }  \Sigma   ( \beta _{\eta_k} ^*  - \beta _{\eta_{k-1}}^* ) 
 \\\nonumber 
 =  &  \frac { ( \beta _{\eta_k} ^*  - \beta _{\eta_{k-1}}^* )^\top  } { \| \beta _{\eta_k} ^*  - \beta _{\eta_{k-1}}^*  \| _2 }  \Sigma   ( \beta _{\eta_k} ^*  - \beta _{\eta_{k-1}}^* )   +  
 \bigg(  \frac { ( \widehat \alpha _1 - \widehat \alpha _2 )^\top  } { \| \widehat \alpha _1 - \widehat \alpha _2 \| _2 } -\frac { ( \beta _{\eta_k} ^*  - \beta _{\eta_{k-1}}^* )^\top  } { \| \beta _{\eta_k} ^*  - \beta _{\eta_{k-1}}^*  \| _2 }  \bigg)   \Sigma   ( \beta _{\eta_k} ^*  - \beta _{\eta_{k-1}}^* )  
 \\\nonumber 
 \ge &   \frac { ( \beta _{\eta_k} ^*  - \beta _{\eta_{k-1}}^* )^\top  } { \| \beta _{\eta_k} ^*  - \beta _{\eta_{k-1}}^*  \| _2 }  \Sigma   ( \beta _{\eta_k} ^*  - \beta _{\eta_{k-1}}^* )    -   2  \frac{  \| \widehat \alpha _1 - \widehat \alpha _2  - ( \beta _{\eta_k} ^*  - \beta _{\eta_{k-1}}^* )  \|  \| \Sigma\| _{\op } \| \beta _{\eta_k} ^*  - \beta _{\eta_{k-1}}^*  \|_2^2     }{ \| \widehat \alpha _1 - \widehat \alpha _2 \| _2  \| \beta _{\eta_k} ^*  - \beta _{\eta_{k-1}}^*  \| _2  }
 \\ 
 \ge 
 & c_x \|\beta _{\eta_k} ^*  - \beta _{\eta_{k-1}}^*  \| _2  - 3 C_x \frac{c_x}{32 C_x }   \|\beta _{\eta_k} ^*  - \beta _{\eta_{k-1}}^*  \| _2 
 \ge \frac{7}{8}c_x  \|\beta _{\eta_k} ^*  - \beta _{\eta_{k-1}}^*  \| _2  . \label{eq:lower bound signal vpcusm}
\end{align}	
	Therefore \Cref{eq:one change point 1d cusum} gives 
	\begin{align*}
		| \widetilde f_{\eta_k }^{a_k , b _k }  (u_k)   | = \sqrt { \frac{  ( \eta_k-a_k)(b_k-\eta_k ) } {b_k-a_k }} \bigg| u_k ^\top  \Sigma   ( \beta _{\eta_k} ^*  - \beta _{\eta_{k-1}}^* ) \bigg| 
	\ge \frac{1}{4} \sqrt {\Delta}   \frac{7}{8}c_x  \|\beta _{\eta_k} ^*  - \beta _{\eta_{k-1}}^*  \| _2 = 
	\frac{7c_x}{32} \sqrt {\Delta}      \kappa_k   ,
	\end{align*} 
	where the last inequality follows from \Cref{eq:spacing of vpcusum} and \Cref{eq:lower bound signal vpcusm}.    Under  Assumption \ref{assume: model assumption} and Assumption \ref{assume: LGS assumption}, for sufficiently large constant $ C_{snr}$, it holds that $\zeta  \le \Delta /4$. Therefore 
	\Cref{eq:spacing of vpcusum} also implies that 
	$ \eta_k \in [ a_ k  +\zeta, b_k -\zeta ]$
	 $$\max_{  a_ k +\zeta  \le t \le b_k  -\zeta   }  | \widetilde f_t^{a_k , b _k } (u_k )    |  \ge | \widetilde f_{\eta_k }^{a_k , b _k }  (u_k )   |  \ge \frac{7c_x}{32} \sqrt {\Delta}      \kappa_k . $$
	This directly gives \Cref{eq:properties of vectors}.

\noindent{\bf Step 2.}
In this step, we will show that   VPWBS$( (s, e] , \{ (a_m,b _m)\}_{m=1}^M  ,\lambda, \tau , \zeta  )$ 
 consistently detect or reject the existence of undetected
change points within $(s, e]$.
 \\
\\
Let $a_m, b_m$ and  $m^*$ be defined as in VPWBS$( (s, e] , \{ (a_m,b _m)\}_{m=1}^M  ,\lambda, \tau , \zeta )$.  Denote $z_i(u_m)= 
(u_m ) ^{\top}
x_i ^{(2) } y_i^{(2) } $ and  $f_i(u_m) =  E \{ z_i(u_m)\} = u_m \Sigma \beta_i^*  $.  Let  $ \widetilde
Z_{t}^{s,e}(u_m)$ and  $\widetilde f_{t}^{s,e} (u_m)$ be the CUSUM statistics of the time series  $z_i(u_m) $ and $f _i(u_m)$, 
respectively.
\\
\\
Suppose there exists a change point $\eta_r\in (s, e ]$ such that $ \min\{ \eta_r
-s, e-\eta_r\} \ge 3\Delta/4 $.    Then, on the
event $\mathcal M$, there exists an interval $(a_k,b_k]$ selected by
VPWBS such that $a_k  \in [\eta_r-3\Delta/4, \eta_r -\Delta/2] $ and  $ b_k \in [\eta_r+\Delta/2, \eta_r +3\Delta/4]$.
Then $[a_k, b_k] \subset [s,e] $ and so  
$$  ( s_k , e_k ] =  ( a_k ,b_ k ]\cap  (s,e] =( a_k , b_ k  ].$$
Since  \Cref{eq:properties of vectors}  in {\bf Step 1} holds for $(a_k, b_k]$,
 we have that 
		\begin{align*}
			 A_ k &  =\max_{  a_k   +\zeta  \le t \le b_ k -\zeta   }  | \widetilde Z_t^{a_k , b _k } (u_ k)    |  \\
			 & \ge \max_{  a_ k +\zeta  \le t \le b_ k -\zeta   }  | \widetilde f_t^{a_k , b _k } (u_ k)    |    -C_1  \sqrt { (\N + 1)\log(n) }  \\
			 & \ge   \frac{7c_x}{32} \sqrt {\Delta} \kappa_r  -  C_1  \sqrt { (\N + 1)\log(n) } 
			 \\
			 & \ge \frac{7c_x}{64}  \sqrt {\Delta} \kappa_r  \ , 
		\end{align*}
	where the first inequality holds on the event 
	$\mathcal A ( \{y_i  ^{(2) } ,  x_i  ^{(2) } \}_{i=1}^n  ,   \{ u_m \}_{m=1}^M ,\xi = C_1 \sqrt { (\N + 1)\log(n) } ) $, the second inequality follows from  \Cref{eq:properties of vectors}, 
	and the last inequality follows from  Assumption \ref{assume: LGS assumption} with sufficiently large constant $C_{snr}$.   Thus for any
	undetected change point $\eta_r$ within $(s, e]$, it holds that
		\begin{align}
			A_{m^*}   = \sup_{1\le m\le  M} A_m \ge A_k  \ge c'   \sqrt {\Delta} \kappa_r       \label{eq:wbsrp size of population}.
		\end{align}
		By  Assumption \ref{assume: LGS assumption} with sufficiently large constant $C_{snr}$,  \Cref{eq:wbsrp size of population} gives 
	 	\begin{align*}
			A_{m^*}    \ge  c'   \sqrt {\Delta} \kappa_r   >  C_\tau \sqrt {  (\N +1) \log(n)} =\tau    . 
		\end{align*}
\
\\
As a result, VPWBS$( (s, e] , \{ (a_m,b _m)\}_{m=1}^M  ,\lambda, \tau , \zeta    )$    correctly
accepts the existence of undetected change points. 
\\
\\
Suppose there does not exist any undetected change points in $(s, e]$. Then for any $(s_m , e_m  ] = (a_m, b_m] \cap (s, e]$, one of the following situations must hold.
	\begin{itemize}
	\item [(a)]	There is no change point within $ (s_m , e_m  ]$;
	\item [(b)] there exists only one change point $\eta_{r} $ within $(s_m, e_m]$ and $\min \{ \eta_{r}-s_m  , e_m -\eta_{r}\}\le \delta_r $; 
	\item [(c)] there exist two change points $\eta_{r} ,\eta_{r+1}$ within $(s_m, e_m]$ and 
	$$  \eta_{r}-s_m \le \delta_r  \quad \text{and} \quad    e_m -\eta_{r+1}  \le \delta_ {r+1}  .$$
	\end{itemize}
The calculations of (c) is provided as the other two cases are similar and simpler.  Note that for any $\|u_m  \|_2 =1$, 
it holds that 
$$ | f _{\eta_{r+1} }  (u_m  )   - f _{\eta_{r+1} +1 } (u_ m )  |  =  |u_m^\top  \Sigma (\beta_{\eta_{r+1}} ^*  -\beta_{\eta_{r+1} + 1} ^*  ) |
\le \| u_  m  \|_2 \| \Sigma\|_\op \|  \beta_{\eta_{r+1}} ^*  -\beta_{\eta_{r+1} + 1} ^* \|\le C_x \kappa_{r+1}  $$
and similarly 
$$ | f _{\eta_{r } }  (u_ m   )   - f _{\eta_{r } +1 }  (u_m )  |   \le C_x \kappa_{r }. $$
By \Cref{lemma:cusum boundary bound} and the assumption that $(s_m ,e _m]$ contains only two change points, it holds that
\begin{align*}\nonumber 
			\max_{s_m \leq t \leq e_m }  |\widetilde{f}^{s_m, e_m }_t (u_m )  | \leq & \sqrt{e_m  - \eta_{r+1}}  | f _{\eta_{r+1} }  (u_m )   - f _{\eta_{r+1} +1 } (u_m ) |  + \sqrt{\eta_r - s_m }  |  f _{\eta_{r } }  (u_m )   - f _{\eta_{r } +1 }  (u_  m   ) | 
 \\
	 \le   & C_2  \sqrt  { \delta_{r+1} }  \kappa_{r+1 } +  C_2\sqrt { \delta_{r }}  \kappa_{r  } \le  C_3 \sqrt {(\N +1 )\log(n) } .
	 \end{align*}
 Therefore  under event $\mathcal A ( \{y_i,  x_i\}_{i=1}^n  ,   \{ u_m \}_{m=1}^M ,\xi = C_1 \sqrt { (\N + 1)\log(n) } ) $,
 $$A_ m : =   \max_{  s_m   +\zeta  \le t \le e_m  -\zeta   }  | \widetilde Z_t^{s_m , e_m  } (u_  m )    | \le  \max_{s_m  +\zeta \leq t \leq e_m -\zeta }  |\widetilde{f}^{s_m, e_m }_t (u_m)  | + C_1 \sqrt {(\N +1) \log(n) } \le C_4 \sqrt {(\N +1) \log(n) } . $$
 So if $ \tau = C_\tau \sqrt {(\N +1) \log(n) }$ for sufficiently large $C_\tau$, it holds that 
 $$A_ m \le  \tau \quad \text{for all } 1\le m  \le M.$$
 As a result, VPWBS$( (s, e] , \{ (a_m,b _m)\}_{m=1}^M  ,\lambda, \tau , \zeta  )$    correctly
reject  if $(s,e]$ contains no undetected change points.

\noindent {\bf Step 3.} 
Assume that there exists a change point $\eta_r\in (s, e]$ such that 
$$ \min\{ \eta_r -s, \eta_r-e\} \ge 3\Delta/4.$$  Let $a_m, b_m$ and $m^*$ be defined as in  VPWBS$( (s, e] , \{ (a_m,b _m)\}_{m=1}^M  ,\lambda, \tau , \zeta  )$. 
\\
\\
To complete the  induction,  it suffices to show that,   there exists a change point $\eta_k\in  (s_{m*},e_{m*}]$ such that 
$ \min\{ \eta_k -s, \eta_k-e\} \ge 3\Delta/4 $ and $|D_{m*}-\eta_k|\le \delta_k$.
\\
\\
Consider the univariate time series $$ z_i (u_{m*} )  = ( u   _  m  ^\top     x_i ^{(2) })  y_i ^{(2) }   , \quad  
f _i  (u_{m*})  =E\{z _i (u_{m*}   ) \}  \quad \text{ for all } 1\le i \le n .$$ Since the
    collection of the change points of the time series $\{f_{i} (u_{m*}
)\}_{i=s_{m*} +1 }^ {e_{m*}}$ is a subset of that of $\{\eta_{k}\}_{k=0}^{K+1}\cap
[s,e]$, we may apply \Cref{lemma:error bound 1d} to the time series $\{z_{i}
(u_{m*})\}_{i=s_{m*}+1 }^{e_{m*}}$ and $\{f_{i} (u_{m*})\}_{i=s_{m*}+1 }^{e_{m*}}$. 
Therefore, it suffices to justify that all the assumptions of 
\Cref{lemma:error bound 1d} hold. 
\\
\\
Let $\zeta =C_\zeta (\N +1) \log(n)$ and $\xi  = C_1 \sqrt { (\N +1) \log(n) }$.  
Observe that  from {\bf Step 2} 
\Cref{eq:wbsrp size of population}, it holds that  
\begin{align*}
			A_{m^*}    \ge c'   \sqrt {\Delta} \kappa_r       \label{eq:wbsrp size of population}.
		\end{align*}
for all $r$ such that $ \min\{ \eta_r
-s, e-\eta_r\} \ge 3\Delta/4 $. So \Cref{eq:wbs size of sample} holds.  
\Cref{eq:wbs noise 1} and  \Cref{eq:wbs noise 2}  are direct consequences of 
$\mathcal A ( \{y_i ^{(2) },  x_i ^{(2) }\}_{i=1}^n  ,   \{ u_m \}_{m=1}^M ,\xi = C_1 \sqrt { (\N + 1)\log(n) } )$  and 
$ \mathcal B ( \{y_i ^{(2) },  x_i ^{(2) }\}_{i=1}^n  ,   \{ u_m \}_{m=1}^M ,  \xi = C_1 \sqrt { (\N + 1)\log(n) }  )$. \Cref{eq:wbs noise} is a direct consequence of Assumptions \ref{assume: model assumption} and \ref{assume: LGS assumption}.
 \\
 \\
Thus, all the conditions in \Cref{lemma:error bound 1d} are met, and we therefore
conclude that there exists a change point $\eta_{k}$ of $\{f_i( u_{m*})\}_{i=s_{m^*}+1}^{e_{m^*}} $, satisfying
\begin{equation}
\min \{e_{m^*}-\eta_k,\eta_k-s_{m^*}\}    >  \Delta /4  , \label{eq:coro wbsrp 1d re1}
\end{equation}
and
\[
		| D_{m*}-\eta_{k}|\le  \max \{ C_3\xi ^2 \kappa_k ^{-2}
	    ,\zeta  \} \le C (\N +1 ) \log(n)\kappa_k ^{-2},
	    \]
	where the last inequality holds because  
	$$ C (\N +1 ) \log(n)\kappa_k ^{-2}\ge C (\N +1 ) \log(n)C_\kappa^{-2} \ge C_\zeta (\N +1 ) \log(n) =\zeta$$
for sufficiently large $C $.
Observe that \\
{\bf i)} The change points
of $\{f_i(u_{m^*})\}_{i=s+1}^e $ belong to $(s, e]\cap \{ \eta_k\}_{k=1}^K$; and
\\
{\bf ii)} \Cref{eq:coro wbsrp 1d re1}  and  $(s_{m^*}, e_{m^*} ]  \subset (s, e]$ imply that
	\[
	\min \{e-\eta_k,\eta_k-s\}  >  \Delta /4  > \frac{C(\N +1) \log(n) }{\kappa^2 } = \delta_{\max }.
	\]
	As discussed in the argument before {\bf Step 1}, this implies that
	$\eta_k $ must be an undetected change point of  $\{\beta_i^* \}_{i=1}^n $.
\end{proof}

	  \subsection{Additional Technical Lemmas}
	  
	 Let  $\{ a_m\}_{m=1}^M,\{b _m\}_{m=1}^M$ be two sequences independently selected at random from $\{1, \ldots, n \}$, and   
	\begin{equation}\label{eq:event M}
		\mathcal{M} =   \bigl\{  \text{For each } k\in \{1,\ldots, K \},   \text{there exist   one } m\in \{ 1, \ldots, M \} \text{ such that } a_m \in \mathcal{S}_k, b_m \in \mathcal{E}_k \bigr\}, 
	\end{equation}
	where $\mathcal S_{k}= [\eta_k-3\Delta/4, \eta_k-\Delta/2 ]$ and $\mathcal
	E_{k}= [\eta_k+\Delta/2, \eta_k+3\Delta/4 ]$.  In the following lemma below, we give a lower bound on the probability of $\mathcal{M}$.   
\begin{lemma}\label{lemma:random interval}
	For the event $\mathcal{M}$ defined in \eqref{eq:event M}, we have
	\[ 
			\mathbb{P}(\mathcal M) \geq 1 -\exp\left\{\log\left(\frac{n}{\Delta}\right) - \frac{M\Delta^2}{16 n^2} \right\}.
		\]
\end{lemma}
	
\begin{proof}
	Since the number of change points are bounded by $n/\Delta $,
		\begin{align*}
			\mathbb{P}\bigl(\mathcal{M}^c\bigr) \leq & \sum_{k=1}^K \prod_{m =1}^M \bigl\{1 - \mathbb{P}\bigl( a _m \in \mathcal{S}_k,  b _m \in \mathcal{E}_k\bigr)\bigr\} 
			\\
				 \leq & K  (1-\Delta^2/(16n ^2))^M \leq ( n /\Delta)  (1 - \Delta^2/(16 n ^2))^M 
				 \\
				 \le & \exp\left\{\log\left(\frac{n}{\Delta}\right) - \frac{M\Delta^2}{16 n^2} \right\} .
		\end{align*}
\end{proof}

	  \subsubsection{Univariate CUSUM Statistics}
We  introduce some notations for one dimensional change point detection and the corresponding CUSUM statistics. Let 
 $\{z _i\}_{i=1}^n, \{f_i\}_{i=1}^n \subset \mathbb R$ be two univariate sequences. We will make the following assumptions.
 
\begin{assumption}[Univariate mean change points]\label{assume:model 1d}
	 Let $\{\eta_k\}_{k=0}^{K+1} \subset \{0, \ldots, n\}$, where $\eta_0=0 $ and $\eta_{K+1}=n$, and 
	$$ f_{\eta_{k-1} +1} = f_{\eta_{k-1}+2} =\ldots = f_{\eta_{k}} \quad \text{for all} \quad 1\le k\le K+1,$$
	Assume
	\begin{align*}
		& \min_{k = 1, \ldots, K+1} ( \eta_k-\eta_{k-1} ) \ge  \Delta > 0,\\
		& 0< |f_{\eta_{k+1}} - f_{\eta_{k}} |:  = \kappa_{k}   \text{ for all }  k = 1, \ldots, K.
	\end{align*}
\end{assumption}

 We also have the corresponding CUSUM statistics over any
 generic interval $[s,e]\subset [1,T]$ defined as
	\begin{align*}
	\widetilde Z_{t}^{s,e} &=\sqrt{\frac{e-t}{(e-s) (t-s)}}\sum_{i=s+1}^{t}z_i- \sqrt{\frac{t-s}{(e-s) (e-t)}} \sum_{i=t+1}^{e} z_i ,\\
	\widetilde f_{t}^{s,e} &=\sqrt{\frac{e-t}{(e-s) (t-s)}}\sum_{i=s+1}^{t}f_i- \sqrt{\frac{t-s}{(e-s) (e-t)}} \sum_{i=t+1}^{e} f_i.
	\end{align*}
Throughout this  section, all of our results are proven by
regarding $\{Z_i\}_{i=1}^T$ and  $\{f_i\}_{i=1}^T$ as two deterministic sequences.
We will frequently assume that $\widetilde f_{t}^{s,e}$ is a good approximation
of $\widetilde Z _{t}^{s,e}$ in ways that we will specify through appropriate
 assumptions.

\begin{lemma}\label{lemma:error bound 1d}
Suppose  Assumption \ref{assume:model 1d} holds.  Let  $[s_0,e_0]$ be an interval with $e_0-s_0\le
C_R\Delta$ and contain at lest one change point $\eta_r$ such that 
\[
\eta_{r-1} \le s_0\le \eta_r \le \ldots\le \eta_{r+q} \le e_0 \le \eta_{r+q+1},
\quad q\ge 0.
\]
 Suppose  that $\min\{ \eta_{p'} -s_0 , e_0 -\eta_{p'} \}\ge \Delta /16$ for
 some $p'$ and let $\kse= \max\{\kappa_p: \min\{ \eta_p -s_0 , e_0 -\eta_p \} \ge \Delta /16\}$.    Let  $[s,e] \subset [s_0,e_0]$ be any generic intervals. 
and  
$$b \in \arg \max_{s < t < e}|\widetilde Z _{t}^{s,e}  | .$$
For some $c_1>0$, $\lambda>0$ and $\delta>0$, suppose that
\begin{align}
&|\widetilde Z  _{b}^{s,e}  |  \ge c_1 \kse \sqrt{\Delta},  \label{eq:wbs size of sample} \\
&\sup_{s +\zeta  \le  t \le e- \zeta } |\widetilde Z _{t}^{s,e}   - \widetilde f_{t}^{s,e} | \le \xi, \quad \label{eq:wbs noise 1}   \text{and}   \\
&\sup_{s_1 < t < e_1} \frac{1}{\sqrt{e_1-s_1}}\left| \sum_{t=s_1+1}^{e_1} (     z _t
-f_t )\right| \le  \xi   \label{eq:wbs noise 2}  \quad \text{for every} \quad
e_1-s_1 \ge \zeta . 
\end{align}
If there exists a sufficiently small $c_2 > 0$ such that
\begin{equation}\label{eq:wbs noise}
\xi \le c_2\kse\sqrt \Delta \quad \text{and} \quad  \zeta  \le c_2\Delta,
\end{equation}
then there exists a change point $\eta_{k} \in (s, e)$  such that 
\[
    \min \{e-\eta_k,\eta_k-s\}  >  \Delta /4  \quad \text{and} \quad 
|\eta_{k} -b |\le \min \{C_3\xi ^2\kappa_k^{-2},\zeta \}.
\]
\end{lemma}
\begin{proof}
This is  Lemma 22 in \cite{wang2017optimal}.
\end{proof} 

\begin{lemma}\label{lemma:cusum boundary bound}
		If $[s, e]$ contain two and only two change points $\eta_r$ and $\eta_{r+1}$, then
		\[
			\sup_{s\leq t \leq e} \left|\widetilde{f}^{s, e}_t\right| \leq \sqrt{e - \eta_{r+1}} \kappa_{r+1} + \sqrt{\eta_r - s} \kappa_r.
		\]
		
\end{lemma}
 \begin{proof}
This is  Lemma 24 in \cite{wang2021optimal}.
\end{proof} 
\subsubsection{Projected CUSUM Statistics}
Given a collection of deterministic vectors $\{ u_m \}_{m=1}^M \in \mathbb R^p  $,  denote 
 $$z_{i  } ( u  _  m  )  =      u   _  m  ^\top     x_i   y_i     \in \mathbb  R\quad \text{ for } 1\le m \le M   \text{ and } 1\le i\le n.   $$
Let  $\widetilde Z^{s ,e }_{t}  (u_m) $ denote the corresponding one-dimensional CUSUM statistics.  
That is  
\begin{align*}
	\widetilde Z^{s ,e }_{t} (u_m)   =  \sqrt \frac{ e-t  }{(e-s) (t-s)  }  \sum_{i=s+1}^t z_i  (u_m) -  \sqrt \frac{   t-s  }{(e-s)(e-t )   }  \sum_{i=t+1}^e  z_i  (u_m)  .
\end{align*}  
Consider the following events 
\begin{align}\label{eq:event A} 
 \mathcal A ( \{y_i,  x_i\}_{i=1}^n  , & \{ u_m \}_{m=1}^M , \xi )  
\\ \nonumber 
=&\Bigg\{ \sup_{1\le m \le M } \sup_{ 0 \le s < t <e \le n }  |  \widetilde Z^{s ,e }_{t} (u_m)   - E (\widetilde Z^{s ,e }_{t} (u_m) )  |\ge \xi , 
\min \{ t-s , e-t \}\ge (\N+1)  \log(n)   \Bigg \} ;
\\\label{eq:event B} 
 \mathcal B ( \{y_i,  x_i\}_{i=1}^n  , & \{ u_m \}_{m=1}^M , \xi )  
\\ \nonumber 
=&\Bigg\{ \sup_{1\le m \le M } \sup_{ 0 \le s < t <e \le n }  \bigg  |   \frac{  \sum_{i=s+1}^e  \{  z _{i} (u_m)   - E (z _{i} (u_m)   ) \} }{\sqrt {e-s }}   \bigg |\ge \xi , 
\min \{ e-s  \}\ge (\N+1)  \log(n)   \Bigg \} .
\end{align}

\begin{lemma}[Deviation Bounds for Variance-Projected CUSUM statistics]
Suppose Assumption \ref{assume: model assumption} {\bf a} holds. Let $\{ u_m\}_{m=1}^M  $ be  a collection of vectors in $\mathbb R^p$ such that 
$ \| u_m\|_2 = 1 $ for  all $m$.  Then there exists an absolute constants $C_1 $ and $C_2$ such that 
\begin{align*}
&P \big    [    \mathcal A ( \{y_i,  x_i\}_{i=1}^n  ,   \{ u_m \}_{m=1}^M , C_1 \sqrt { (\N + 1)\log(n) }  )     \big ] \ge 1- CMn^{-3} , \text{ and } 
\\
&P \big    [    \mathcal B ( \{y_i,  x_i\}_{i=1}^n  ,   \{ u_m \}_{m=1}^M , C_1 \sqrt { (\N + 1) \log(n) }  )     \big ] \ge 1- CMn^{-3} .
\end{align*}

\end{lemma}
\begin{proof}
The deviation bounds can be established by standard sub-Exponential tail bounds. The analysis for   the event $\mathcal A$ will be provided, as the analysis for event $\mathcal B$ is exactly the same.
\\
\\
{\bf Step 1.}
Note that 
\begin{align*}   \widetilde Z^{s ,e }_{t} (u_m)   - E (\widetilde Z^{s ,e }_{t} (u_m) 
=
 \sum_{i=s+1}^e   b_i     \Big [   u   _  m  ^\top     x_i   y_i   - E\{   u   _  m  ^\top     x_i   y_i  \}  \Big] \ , 
\end{align*} 
where 
$$b_i = \begin{cases}
\sqrt { \frac{e-t }{ (e-s) (t-s )} }  & \text{when}  s+1\le i \le t,
\\
- \sqrt { \frac{t-s  }{ (e-s) (e- t )} } & \text{when}  t+1\le i \le e.
\end{cases} $$ 
 \
 \\
Note that
 $    u   _  m  ^\top     x_i   y_i   = u_m^\top  x_i ( x_i^\top \beta_i  ^* + \varepsilon_i)$ where  $x_i^\top \beta_i  ^*   +\varepsilon_i$ is centered Gaussian with
 $$ Var( x_i^\top \beta_i^*   + \varepsilon_i)  = ( \beta_i ^* ) ^\top  \Sigma \beta_i ^*  + \sigma_\varepsilon^2 \le  \N  C_x +  \sigma_\varepsilon^2 , $$  and $u_m x_i $ is  centered Gaussian with 
with $$Var(u_m ^\top x_i) = u_m ^\top  \Sigma  u_m  \le C_x , $$
where $\|u_m\|_2^2=1$ is used in the last inequality. 
So 
$ u   _  m  ^\top     x_i   y_i   $ is sub-Exponential with parameter $\N  C_x^2  +\sigma_\varepsilon^2C_x   $.  
In addition, note that 
$$ \sum_{i=s+1}^e b_i^2 = 1 \quad \text{and} \quad  |b_i | \le   ( \N \log(n) )^{-1/2} .$$
So by sub-Exponential tail bound, 
it holds that 
$$ P  \bigg(  \bigg| \sum_{i=s+1}^e   b_i     \Big [   u   _  m  ^\top     x_i   y_i   - E\{   u   _  m  ^\top     x_i   y_i  \}  \Big]   \bigg| \ge \delta  \bigg) \le 2 \exp\bigg(  -c \min\bigg \{  \frac{\delta^2 }{ \N C_x  ^2 +   C_x   \sigma_\varepsilon^2  } ,  \frac{ \delta \sqrt {  (\N+1)  \log(n) } }{ \sqrt { \N  C_x  ^2 +   C_x \sigma_\varepsilon^2 }  }\bigg\}  \bigg)  .$$
So by picking $\delta = C_\delta \sqrt { (\N + 1) \log(n)   }  $ for  sufficiently large constant $C_\delta$, it holds that with probability at most 
$ 1-n^{-6}$,  
$$  \bigg|  \widetilde Z^{s ,e }_{t} (u_m)   - E (\widetilde Z^{s ,e }_{t} (u_m)  \bigg| =  \bigg| \sum_{i=s+1}^e   b_i     \Big [   u   _  m  ^\top     x_i   y_i   - E\{   u   _  m  ^\top     x_i   y_i  \}  \Big]   \bigg| \ge C_\delta \sqrt { (\N + 1) \log(n)   }  . $$
Since there are at most $n^2$ possible choice for $(s,e] \subset (0,n]  $, 
a straightforward   union bound argument  shows that 
$$P \big    [    \mathcal A ( \{y_i,  x_i\}_{i=1}^n  ,   \{ u_m \}_{m=1}^M , C_1 \sqrt {(\N + 1)  \log(n) }  )     \big ] \ge 1- CMn^{-3}.$$
\end{proof}

\begin{center}
	\begin{figure}[h]
		\includegraphics[scale=0.45]{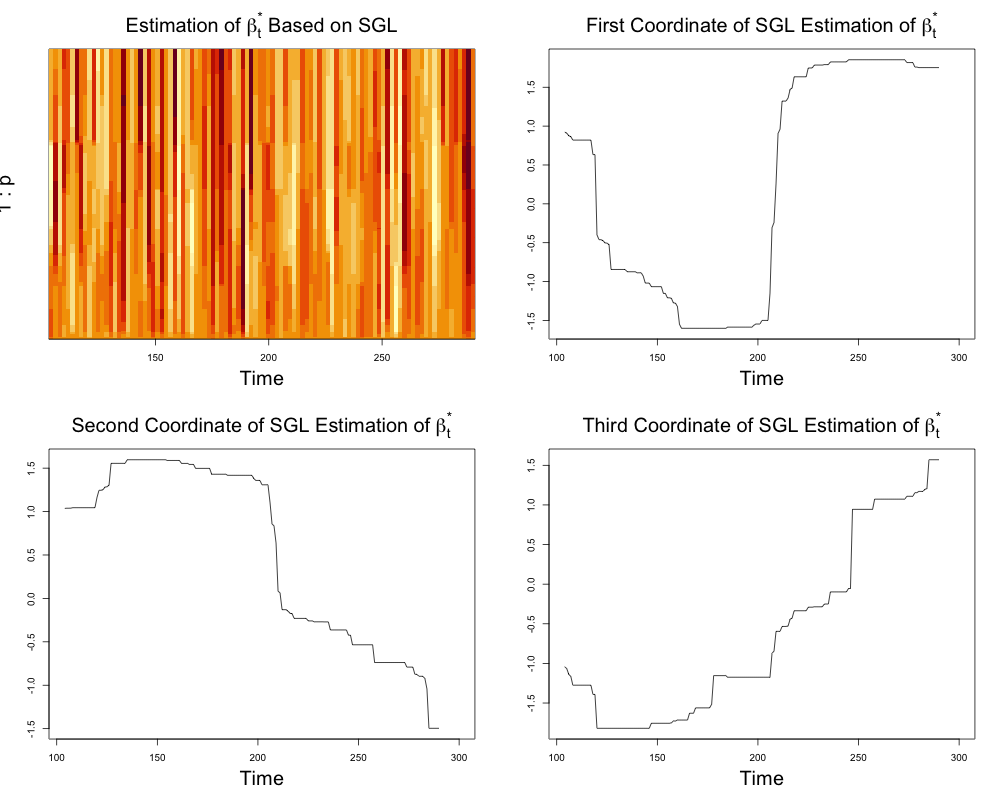} 
		\caption{Plots of SGL estimator $\{\widehat \beta_t\}_{t=1}^n$ in \eqref{eq:SGL full}. The data $\{x_t,y_t\}_{t=1}^n$ is the same as the one used to generate the illustration of VPWBS in Figure \ref{figure:projected}, where we have $n=300, p=100$ and that two change-points  are  at $\eta_1=100$ and $\eta_2=200$ with change size $\kappa= 1.6\sqrt {40}$. For better comparison with \Cref{figure:projected}, we plot the estimated $\widehat \beta_t$ for $t=105,\cdots,290$. The true coefficient $\{\beta_t^*\}_{t=105}^{290}$ contains a single change-point at $\eta_2=200$.} \label{fig:sgl}
	\end{figure} 
\end{center}